\documentclass[11pt]{amsart}
\usepackage{amsmath,amssymb,amsfonts,amsthm,amscd,indentfirst}
\usepackage{amsmath,amssymb,amsfonts,amsthm,amscd}
\usepackage{indentfirst}
\usepackage{hyperref}
\usepackage{mathrsfs}
\usepackage{marginnote}

\numberwithin{equation}{section}

\newtheorem{prop}{Proposition}[section]
\newtheorem{theo}[prop]{Theorem}
\newtheorem{lemm}[prop]{Lemma}
\newtheorem{coro}[prop]{Corollary}
\newtheorem{rema}[prop]{Remark}

\newtheorem{defi}[prop]{Definition}

\def\and{\quad{\rm and}\quad}

\def\<{\langle}
\def\>{\rangle}

\usepackage{xcolor}
\usepackage{graphicx}

\newtheorem{manualtheoreminner}{Proposition}
\newenvironment{manualtheorem}[1]{%
	\renewcommand\themanualtheoreminner{#1}%
	\manualtheoreminner
}{\endmanualtheoreminner}

\title{On the $\sigma_2$-Nirenberg problem on $\mathbb{S}^2$}
\allowdisplaybreaks

\begin{document}
	\author[YanYan Li]{YanYan Li}
	\address{Department of Mathematics, Rutgers University, Hill Center, Busch Campus, 110 Frelinghuysen Road, Piscataway, NJ 08854, USA.}
	\email{yyli@math.rutgers.edu}
	
	\author[Han Lu]{Han Lu}
	\address{Department of Mathematics, Rutgers University, Hill Center, Busch Campus, 110 Frelinghuysen Road, Piscataway, NJ 08854, USA.}
	\email{hl659@math.rutgers.edu}
	
	\author[Siyuan Lu]{Siyuan Lu}
	\address{Department of Mathematics and Statistics, McMaster University, 1280 Main Street West, Hamilton, ON, L8S 4K1, Canada.}
	\email{siyuan.lu@mcmaster.ca}
	
	\maketitle

	\begin{abstract}
We establish theorems on the existence and 
compactness of solutions to the $\sigma_2$-Nirenberg problem on the standard sphere
$\mathbb S^2$. A first significant ingredient, a Liouville type theorem for the associated fully nonlinear M\"obius invariant elliptic equations,
  was established in an earlier paper of ours. Our proof of the existence and compactness results requires a number of additional crucial  ingredients which we prove in this paper: A Liouville type theorem for the associated fully nonlinear M\"obius invariant degenerate elliptic equations, a priori estimates of first and second order derivatives of  solutions to the $\sigma_2$-Nirenberg problem, and a B\^ocher type theorem for the associated fully nonlinear M\"obius invariant elliptic equations. Given these results, we are able to complete a fine analysis of a sequence of blow-up solutions to the $\sigma_2$-Nirenberg problem. In particular, we prove that there can be at most one blow-up point for such a blow-up sequence of solutions. This, together with a Kazdan-Warner type identity, allows us to prove $L^\infty$  a priori estimates for  solutions of the $\sigma_2$-Nirenberg problem under some simple generic hypothesis. The higher derivative estimates then follow from classical estimates of Nirenberg and Schauder. In turn, the  existence of solutions to the $\sigma_2$-Nirenberg problem is obtained by an application of the by now standard degree theory for second order fully nonlinear elliptic operators.

	\end{abstract}
	
	\section{Introduction}
	The Nirenberg problem, raised by Nirenberg in the years 1969-1970,
 asks to identify functions $K$ on the two-sphere $\mathbb S^2$ for which there exists a metric $\tilde g$ on $\mathbb S^2$ conformal to the standard metric $g$ such that $K$ is the Gaussian curvature of $\tilde{g}$. Naturally, this problem extends to higher dimensions with the Gaussian curvature replaced by the scalar curvature.

	There has been vast literature on the Nirenberg problem and related ones and
it would be impossible to mention here all works in this area. One significant aspect most directly related to this paper is the fine analysis of blow-up (approximate) solutions or the compactness of the solution set. These were studied in \cite{BC, CGY0, CY, CY2, CLin, Han, L95, L96}, and related references. For more recent and further studies, see \cite{JLX, MM},  and related references. For $n\geq 3$ and $k\geq 2$, the $\sigma_k$-Nirenberg problem was studied in \cite{CHY, LNW1, LNW2}.

	In this paper, we are interested in the existence and compactness of solutions of a nonlinear version of the Nirenberg problem on the standard sphere $(\mathbb{S}^2,g)$. This equation has similar structures to the $\sigma_k$-Yamabe and $\sigma_k$-Nirenberg problems in higher dimensions.
	
	Throughout this paper, we use $(\mathbb{S}^2,g)$ to denote the standard two sphere. On $(\mathbb{S}^2,g)$, for a conformal metric $g_u=e^ug$, let
	\begin{align}\label{Agu}
	A_{g_u}:=-\nabla_{g}^2u+\frac{1}{2}du\otimes du-\frac{1}{4}|\nabla_{g} u|^2g+K_{g} g,
	\end{align}
	where $K_{g}\equiv 1 $ is the Gaussian curvature of the metric $g$.
	
	For $\lambda=(\lambda_1,\lambda_2)\in\mathbb{R}^2$, let $\sigma_1(\lambda):=\lambda_1+\lambda_2$ and $\sigma_2(\lambda):=\lambda_1\lambda_2$ be the elementary symmetric functions. We use $\lambda(g_u^{-1}A_{g_u})$ to denote the eigenvalues of $g_u^{-1}A_{g_u}$, and $\sigma_k(g_u^{-1}A_{g_u})$ to denote $\sigma_k(\lambda(g_u^{-1}A_{g_u}))$ for $k=1,2$. Note that $\sigma_1(g_u^{-1}A_{g_u})=2K_{g_u}$.

	We study the equation
	\begin{align}\label{eqn}
	\sigma_2(g_u^{-1}A_{g_u})=K(x),\quad \lambda(g_u^{-1}A_{g_u})\in\Gamma_2\quad\text{on }\mathbb{S}^2,
	\end{align}
	where
	\begin{align*}
	\Gamma_2:=\{(\lambda_1,\lambda_2):\lambda_1>0, \lambda_2>0\}
	\end{align*}
	is the first quadrant.

	For a positive function $K$ in $C^2(\mathbb{S}^2)$ satisfying the nondegeneracy condition
	\begin{align}\label{Knondege}
	|\nabla K|_{g}+|\Delta K|_{g}>0 \text{ on }\mathbb{S}^2,
	\end{align}
	we define the sets
	\begin{align*}
	\text{Crit}_{+}(K)=\{x\in\mathbb{S}^2|\nabla_{g}K(x)=0,\Delta_{g}K(x)>0\},\\
	\text{Crit}_{-}(K)=\{x\in\mathbb{S}^2|\nabla_{g}K(x)=0,\Delta_{g}K(x)<0\}.
	\end{align*}
	Set $\deg(\nabla K, \text{Crit}_{-}(K)):=\deg(\nabla K,O,0)$, the Brouwer degree, where $O$ is any open subset of $\mathbb{S}^2$ containing $\text{Crit}_{-}(K)$ and disjoint from $\text{Crit}_{+}(K)$. By (\ref{Knondege}), this is well-defined.
	
	For any $K$ satisfying (\ref{Knondege}) and having only isolated nondegenerate critical points, 
	\begin{align*}
	\deg(\nabla K, \text{Crit}_{-}(K))=\sum_{\bar{x}\in\mathbb{S}^2,\nabla K(\bar{x})=0,\Delta K(\bar{x})<0}(-1)^{i(\bar x)}
	\end{align*}
	where $i(\bar x)$ denotes the number of negative eigenvalues of $\nabla^2 K(\bar x)$.
	
	For an introduction
	to degree theories, see e.g. \cite[Chapter 1]{N1}.

	The first main theorem in this paper is the following existence and compactness result for equation (\ref{eqn}).
	
	\begin{theo}\label{Niren}
	Let $(\mathbb{S}^2,g)$ be the standard two sphere, and let $K$ be a positive function in $C^2(\mathbb{S}^2)$ satisfying the nondegeneracy condition (\ref{Knondege}). Then there exists a positive constant $C$ depending only on $K$, such that
		\begin{align}\label{C2norm}
		\|u\|_{C^2(\mathbb{S}^2)}\leq C, \text{ for all $C^2$ solutions $u$ of equation (\ref{eqn}).}
		\end{align}
		Moreover, if $\deg(\nabla K, \text{Crit}_{-}(K))\neq 1$, then (\ref{eqn}) admits a solution.
	\end{theo}

\begin{rema}
	See Proposition \ref{C0esti} for more detailed dependence of $C$ on $K$.
\end{rema}

\begin{rema}
		If $K\in C^{2,\alpha}(\mathbb{S}^2)$, $0<\alpha<1$, and $\mathcal{O}$ is a bounded open subset of $C^{4,\alpha}(\mathbb{S}^2)$ which contains all solutions of (\ref{eqn}), then
		\begin{align*}
		\deg(\sigma_2(g_u^{-1}A_{g_u})-K,\mathcal{O},0)=-1+\deg(\nabla K, \text{Crit}_{-}(K)).
		\end{align*}
		Here the degree on the left hand side is the degree for second order nonlinear elliptic operators defined in \cite{L89}.
	\end{rema}

\begin{rema}
	Such results for the $\sigma_k$-Nirenberg problem was proved in \cite{CHY} on $\mathbb S^4$ for $\sigma_2$; in \cite{LNW1} on $\mathbb{S}^n$ for $\sigma_k$,  $n\geq 3$ and $n/2\leq k\leq n$; and in \cite{LNW2} on $\mathbb S^n$ for $\sigma_k$, $2\leq k< n/2$ and for axisymmetric functions $K$.
\end{rema}

		The following result is a Kazdan-Warner type identity.
	\begin{theo}\label{KW}
		Let $X$ be a conformal Killing vector field on $(\mathbb{S}^2,g)$, and let $g_u=e^ug$ be a conformal metric to $g$ on $\mathbb{S}^2$, where $u$ is a smooth function on $\mathbb{S}^2$, then
		\begin{align}\label{KazdanWarner}
		\int_{\mathbb{S}^2} X(\sigma_2({g_u}^{-1}A_{g_u}))e^u dV_g=0.
		\end{align}
	\end{theo}
	
	\begin{rema}
		 In dimensions $n\geq 3$, such results were proved in \cite{Han2} and \cite{Via2}, see also \cite{GHL}. For $\sigma_1$ instead of $\sigma_2$ in (\ref{KW}), it is the well-known Kazdan-Warner type identity for the Nirenberg problem. See \cite{BEzin} and \cite{KW}.
	\end{rema}

\begin{rema}
	 Note that Theorem \ref{KW} can also be obtained using \cite[Theorem 2.11]{GO}. In this case, we can check that $\sigma_2(g^{-1}A_g)$ is normally conformally variational (see \cite[Definition 2.10]{GO}) on $\mathbb{S}^2$.
\end{rema}

	As usual, there is a necessary condition for the existence of solutions of equation (\ref{eqn}). We say that a function $K$ on $\mathbb{S}^2$ satisfies the Kazdan-Warner type condition if there exists some positive $C^2$ function $f$ on $\mathbb{S}^2$ satisfying
	\begin{align*}
	\int_{\mathbb{S}^2}X(K)fdV_{g}=0,
	\end{align*}
	for any conformal Killing vector field $X$ on $\mathbb{S}^2$.

	\begin{theo}\label{nonexist}
		If $K$ does not satisfy the Kazdan-Warner type condition, then there is no $C^2$ solution to equation (\ref{eqn}).
	\end{theo}

	Theorem \ref{nonexist} is a corollary of Theorem \ref{KW}.
	
	For example, if $K(x)=2+x_3$, then (\ref{eqn}) has no $C^2$ solution. 
	
	\medskip
	
	More generally, other than the $\sigma_2$-equation, we are interested in equations:
	\begin{equation}\label{feqn}
	f(\lambda(g_u^{-1} A_{g_u})) =K(x),\quad \lambda(g_u^{-1}A_{g_u})\in\Gamma,\quad \mbox{on}\
	\mathbb{S}^2,
	\end{equation}
	where the definition of $f$ and $\Gamma$ are given below.
	\medskip
	
	Let 
	\begin{align}\label{gamma2}
	\Gamma\text{ be an open convex symmetric}\text{ cone in }\mathbb{R}^2\text{ with vertex at the origin,}
	\end{align}
	and
	\begin{align}\label{gamma}
	\Gamma_2\subset \Gamma\subset\Gamma_1,
	\end{align}
	where $\Gamma_1:=\{(\lambda_1,\lambda_2):\lambda_1+\lambda_2>0\}$ 
	and $\Gamma_2:=\{(\lambda_1,\lambda_2):\lambda_1>0, \lambda_2>0\}$. 
	Here, $\Gamma$ being symmetric means that $(\lambda_1, \lambda_2)\in \Gamma$ 
	implies $(\lambda_2, \lambda_1)\in \Gamma$.
	Also, a function $f$ defined on $\Gamma$ is said to be symmetric if 
	$f(\lambda_1, \lambda_2)\equiv f(\lambda_2, \lambda_1)$.
	
	It is not difficult to see that $\Gamma$ satisfies  (\ref{gamma2}) and (\ref{gamma})  if 
	and only if $\Gamma=\Gamma_p$ for some $1\le p\le 2$ where
	\begin{equation*}
	\Gamma_p:= \{ \lambda=(\lambda_1, \lambda_2)\ :\
	\lambda_2>(p-2)\lambda_1,\ 
	\lambda_1>(p-2)\lambda_2\}.
	\label{equiv}
	\end{equation*}
	Note that the above definition of $\Gamma_1$ and $\Gamma_2$ is consistent with 
	earlier definitions.

	Let $\Gamma=\Gamma_p$, $1\leq p\leq 2$, and consider
	\begin{align}
	&f\in C^1(\Gamma)\cap C^0(\bar{\Gamma})\text{ is symmetric},\label{f1}\\
	&f\text{ is homogeneous of degree }1,\label{f2}\\
	&f>0,\quad f_{\lambda_i}:=\frac{\partial f}{\partial \lambda_i}>0\text{ in }\Gamma,\quad f\big|_{\partial\Gamma}=0,\label{f3}\\
	&f\text{ is concave in }\Gamma.\label{f5}\\
	&\sum_{i=1}^{n}f_{\lambda_i}\geq \delta\quad\text{in }\Gamma\text{ for some }\delta>0.\label{f4}
	\end{align}

	For $(f, \Gamma)=(\sigma_1, \Gamma_1)$, 
	problem (\ref{feqn}) is the Nirenberg problem. Theorem \ref{Niren} is for $(f, \Gamma)=(\sigma_2^{\frac 12}, \Gamma_2)$. 
	
	\medskip

	In order to prove Theorem \ref{Niren}, as well as to study the more general equation (\ref{feqn}), a number of ingredients and estimates are needed. The first analytical ingredient is a Liouville type theorem for $\sigma_2(\lambda(A^u))=1$, where
	\begin{align*}
	A^u=-\frac{\nabla^2u}{e^u}+\frac{1}{2}\frac{du\otimes du}{e^u}-\frac{1}{4}\frac{|\nabla u|^2}{e^u}I.
	\end{align*}
	When rescaling appropriately a blow-up sequence of solutions of (\ref{eqn}), we are led to an entire solution of $\sigma_2(\lambda(A^u))=1$ on $\mathbb{R}^2$.
	
	Equations $f(\lambda(A^u))=1$, which we call M\"obius invariant equations, is naturally associated with $A^u$. A Liouville type theorem for the M\"obius invariant equations was established in our previous paper \cite{LLL2}.

	Other ingredients and estimates, which are described below, are also needed in analyzing a sequence of blow-up solutions and giving fine asymptotic profile of such blow-up solutions.

	\medskip

	 The following is a Liouville type theorem for  $f(\lambda(A^u))=0$.
	\begin{theo}\label{dege1}
		Let $\Gamma=\Gamma_p$ for some $1<p\leq 2$, and let $u$ be a (continuous) viscosity solution of
		\begin{align*}
		\lambda(A^u)\in\partial\Gamma\quad\text{in}\ \mathbb{R}^2\backslash\{0\}.
		\end{align*}
		Then $u$ is locally Lipschitz in $\mathbb{R}^2\backslash\{0\}$ and radially symmetric about the origin. Moreover, $u(x)$ is monotonically nonincreasing in $|x|$.
	\end{theo}

\begin{coro}\label{corliou}
	Let $\Gamma=\Gamma_p$ for some $1<p\leq 2$, and let $u$ be a (continuous) viscosity solution of
	\begin{align*}
	\lambda(A^u)\in\partial\Gamma\quad\text{in}\ \mathbb{R}^2.
	\end{align*}
	Then $u\equiv \text{constant}$ in $\mathbb{R}^2$.
\end{coro}
	In dimensions $n\geq 3$, such results were proved in \cite{L07, L09} for locally Lipschitz $u$, $\Gamma_n\subset\Gamma\subset\Gamma_1$. In fact, as proved in \cite{LNW}, a continuous viscosity solution of such equations is automatically locally Lipschitz. Therefore, the results in \cite{L07,L09} hold for continuous viscosity solutions. Note that a first such result was proved in \cite{CGY2} for $u\in C_{loc}^{1,1}$, $n=4$, and $\Gamma=\Gamma_2$,

	If $\Gamma=\Gamma_1$, the equation $\lambda(A^u)\in\partial\Gamma$ becomes $\Delta u=0$. Corollary \ref{corliou} can be viewed as a nonlinear extension of the classical Liouville theorem: A nonnegative harmonic function in $\mathbb{R}^2$ is a constant. However, there is no sign condition assumed on $u$ in Corollary \ref{corliou}
	
	Equation $\lambda(A^u)\in\partial\Gamma$ is sometimes equivalently stated as $f(\lambda(A^u))=0$ for $f$ defined on $\bar{\Gamma}$ satisfying (\ref{f1}) and (\ref{f3}). One example is $\Gamma=\Gamma_2$, then the equation becomes $\det(A^u)=0$ together with semi-positive definiteness of $A^u$.

	We then prove the following local derivatives estimates for general $(f,\Gamma)$. For such equations, it is delicate to prove the local gradient estimates of $u$ under the assumption that $u$ is bounded from above, while it is more standard to prove the second derivatives estimates of $u$ under the assumption that $u$ is $C^1$ bounded.
	
	\begin{theo}\label{grad2}
		Let $\Gamma=\Gamma_p$ for some $1<p\leq 2$, $f$ satisfy (\ref{f1})-(\ref{f3}) and (\ref{f4}), $K$ be a $C^1$ positive function in $B_r\subset\mathbb{R}^2$, and let $u\in C^3(B_r)$ satisfy
		\begin{align}\label{feqnR2}
		f(\lambda(A^u))=K,\ \lambda(A^u)\in \Gamma\ \text{ in } B_{r}.
		\end{align}
		Then
		\begin{align*}
		|\nabla u|\leq C\quad in\ B_{r/2},
		\end{align*}
		for some constant $C$ depending only on $r$, $(f,\Gamma)$, and upper bounds of $\sup_{B_r}u$ and $\|K\|_{C^1(B_r)}$.
	\end{theo}
	The above local gradient estimates do not need the concavity assumption (\ref{f5}). Note that if $\Gamma=\Gamma_p$ for some $1\leq p\leq 2$ and $f$ satisfies (\ref{f1})-(\ref{f3}) and (\ref{f5}), then (\ref{f4}) follows.
	
	We also prove the local derivatives estimate for equation $\lambda(A^u)\in\partial\Gamma$.
	\begin{theo}\label{grad3}
		Let $\Gamma=\Gamma_p$ for some $1<p\leq 2$, and let $u$ be a (continuous) viscosity solution of $\lambda(A^u)\in\partial\Gamma$ in $B_1$. Then for every $0<\epsilon<1$, there exists a constant $C$ depending only on $\Gamma$ and $\epsilon$, such that
		\begin{align*}
		|\nabla u|\leq C\ a.e.\ in\ B_{1-\epsilon}.
		\end{align*}
	\end{theo}
	
	We then establish local $C^2$ estimates.
	\begin{theo}\label{C2}
		Let $\Gamma=\Gamma_p$ for some $1<p\leq 2$, $f$ satisfy (\ref{f1})-(\ref{f3}) and (\ref{f5}), $K$ be a $C^2$ positive function in $B_r\subset\mathbb{R}^2$, and let $u\in C^4(B_r)$ satisfy
		\begin{align*}
		f(\lambda(A^u))=K,\ \lambda(A^u)\in \Gamma\ \text{ in } B_{r}.
		\end{align*}
		Then
		\begin{align*}
		|\nabla^2 u|\leq C\quad in\ B_{r/2},
		\end{align*}
		for some constant $C$ depending only on $r$, $(f,\Gamma)$, and upper bounds of $\sup_{B_{r}}u$ and $\|K\|_{C^2(B_{r})}$.
	\end{theo}

The general local gradient estimates in dimensions $n\geq3$ were derived in \cite{L09} using blow-up analysis and the Liouville theorem for degenerate equations together with Bernstein type arguments. Note that the local gradient estimates for $f=\sigma_k^{1/k}$ and $\Gamma=\Gamma_k$ were proved in \cite{GW}; see also \cite{GLW, GW3, LL1, STW, XJWang} for efforts in achieving further generality. The local $C^2$ estimates in dimensions $n\geq3$ were discussed in \cite{CGY2, Chen, GW, LL1}.

\medskip
	The following theorems are B\^ocher type theorems, which characterize the asymptotic behavior of solutions near isolated singularities. They are for $\Gamma=\Gamma_p$, $1<p\leq 2$. For $1<p<2$, the equation is $\lambda(A^u)\in\partial\Gamma$. For $p=2$, the equation is $\lambda(A^u)\in\bar\Gamma$, i.e. supersolutions to $\lambda(A^u)\in\partial\Gamma$. Note that for $p=1$, the equation $\lambda(A^u)\in\partial\Gamma$ is $\Delta u=0$, and additional assumption is needed for the B\^ocher theorem.
	\begin{theo}\label{Bocher1}
		If $\Gamma=\Gamma_2$, let $u\in LSC(B_1\backslash\{0\})\cap L_{loc}^{\infty}(B_1\backslash\{0\})$ be a viscosity supersolution of $\lambda(A^u)\in\bar\Gamma$ in $B_1\backslash\{0\}$. Then either $u$ can be extended to a function in $C_{loc}^{0,1}(B_1)$ or $u=-4\ln |x|+C$ for some constant $C$, and $A^u=0$. Moreover, in the former case, there holds
		\begin{align*}
		\|w\|_{C^{0,1}(B_{1/2})}\leq C(\Gamma)\max_{\partial B_{3/4}}w
		\end{align*}
		where $w=e^{-u/4}$. 
	\end{theo}
	\begin{theo}\label{Bocher2}
		Let $\Gamma=\Gamma_p$ for some $1<p<2$, and let $u\in C_{loc}^{0}(B_1\backslash\{0\})$ be a viscosity solution of $
		\lambda(A^u)\in\partial\Gamma \text{ in }B_1\backslash\{0\}$. Then either $u$ can be extended to a function in $C_{loc}^{0,p-1}(B_1)$, or
		\begin{align*}
		u=\dfrac{4(2-p)}{p-1}\ln(r^{-(p-1)/(2-p)}+\mathring w)+a,
		\end{align*}
		where $a=\sup_{B_1\backslash\{0\}}(u(x)+4\ln |x|)<+\infty$, and $\mathring w\in L_{\text{loc}}^\infty(B_1)$ is a nonpositive function satisfying
		\begin{align*}
		\min_{\partial B_r} \mathring w\leq \mathring w\leq \max_{\partial B_r} \mathring w \text{ in }B_r\backslash\{0\},\quad \forall\ 0<r<1.
		\end{align*}
		Moreover, in the former case, there holds
		\begin{align*}
		\|w\|_{C^{0,p-1}(B_{1/2})}\leq C(\Gamma)\max_{\partial B_{3/4}}w,
		\end{align*}
		where $w=\exp(-\frac{p-1}{4}u)$.
	\end{theo}
	\medskip

	The B\^ocher type theorems for equation $\lambda(A^u)\in\partial\Gamma$ in $B_1\backslash\{0\}$ when $n\geq 3$ were established in \cite{LN1}. The behaviors of our solutions are quite different from the results in \cite{LN1}.
	
	In the case of non-degenerate  elliptic equation $\sigma_k(\lambda(A^u))=1$ for $n\geq3$, the local behavior near isolated singularity was studied in \cite{CGS} for $k=1$ and in \cite{HLT} for $2\leq k\leq n$. They proved that $u(x)=u_*(|x|)(1+O(|x|^\alpha))$ where $u_*$ is some radially symmetric solution of $f(\lambda(A^u))=1$ on $\mathbb{R}^n\backslash\{0\}$ and $\alpha$ is some positive number. See \cite{HLL} for expansions to arbitrary orders.  See also \cite{WYP} for expansions of solutions of conformal quotient equations.

	\medskip

	Recall that when dimension $n\geq 3$, the existence of solutions of the $\sigma_k$-Yamabe problem has been proved for $k\geq n/2$, $k=2$ or when $(M,g)$ is locally conformally flat, the compactness of the set of solutions has been proved for $k\geq n/2$ when the manifold is not conformally equivalent to the standard sphere $-$ they were established in \cite{CGY1,GeW,GW2,GV,LL1,LN2,STW}. For more recent works on $\sigma_k$-Yamabe type problems, see for example \cite{AbEs,BCE,BoSheng,Case,CaseW1,CaseW2,DN,FangW,FangW2,GLN,GS,HLL,HqLyc,He,HXZ,JS1,JS2,JS3,LN3,LN4,LW,Santos,Sui,T2} and references therein. However, there are still many challenging open problems on general compact Riemannian manifolds - the compactness remains open for $2\leq k\leq n/2$ and the existence remains open for $2<k<n/2$. One motivation of studying the equations in dimension two is to gain insights and inspirations into solving the above mentioned open problems in dimensions $n\geq 3$.
	
	The strategies of the proofs of our main theorems are described as follows. Liouville type theorem Theorem \ref{dege1} is established using comparison principles and asymptotic behavior of solutions. For local gradient estimates Theorem \ref{grad2}, we first prove it using Bernstein type arguments assuming in addition $u$ is also bounded from below, then establish the result using blow-up analysis and Theorem \ref{dege1}. Theorem \ref{grad3} is proved in a similar way. Local $C^2$ estimates Theorem \ref{C2} is obtained by Bernstein type arguments. In the proof of B\^ocher type theorems Theorem \ref{Bocher1} and \ref{Bocher2}, we first classify all the radially symmetric (continuous) viscosity solutions in any annulus $\{0\leq a\leq |x|\leq b\leq \infty\}$, then establish the results with the help of a comparison principle.
	
	In the proof of existence and compactness result Theorem \ref{Niren}, we first prove the compactness part. Since $C^1$ and $C^2$ estimates are already established, the only issue left is a $C^0$ estimate. We first analyze the behavior of a sequence of blow-up solutions and prove that every sequence of solutions cannot blow up at more than one point with the help of the Liouville type theorem for $\sigma_2(\lambda(A^u))=1$. Then we obtain an optimal decay estimate, where the B\^ocher type theorem is used. Next we use a Kazdan-Warner type identity together with the nondegeneracy condition on $K$ and the above one point blow-up behavior to prove the $C^0$ estimate. The existence part is proved thanks to the degree theory and compactness of solutions.

	The rest of our paper is organized as follows. In Section \ref{degliou}, we establish the Liouville type theorem Theorem \ref{dege1} and the local derivatives estimates Theorem \ref{grad2}, Theorem \ref{grad3} and Theorem \ref{C2}. B\^ocher type theorems Theorem \ref{Bocher1} and Theorem \ref{Bocher2} are proved in Section \ref{bochertype}. The existence and compactness theorem Theorem \ref{Niren} is proved in Section \ref{existence}. Three calculus lemmas are given in Appendix \ref{appendix} for readers' convenience.

	\section*{Acknowledgements}
	The first named author's research was partially supported by NSF Grants DMS-1501004, DMS-2000261, and Simons Fellows Award 677077. The second named author's research was partially supported by NSF Grants DMS-1501004, DMS-2000261. The third named author's research was partially supported by NSERC Discovery Grant.

	\medskip

\section{Liouville type theorems and Local Estimates}\label{degliou}

\medskip

\subsection{Preliminaries}
\ 

\subsubsection{Viscosity solutions}

\

In order to introduce the definition of viscosity solution, let us first define the set of upper semicontinuous and lower semicontinuous functions.

For any set $S\subset \mathbb{R}^2$, we use $USC(S)$ to denote the set of functions $u:S\rightarrow \mathbb{R}\cup \{-\infty\}$, $u\neq -\infty$ in $S$, satisfying
\begin{align*}
\limsup_{x\rightarrow x_0}u(x)\leq u(x_0),\quad \forall x_0\in S.
\end{align*}

Similarly, we use $LSC(S)$ to denote the set of functions $u:S\rightarrow \mathbb{R}\cup \{+\infty\}$, $u\neq +\infty$ in $S$, satisfying
\begin{align*}
\liminf_{x\rightarrow x_0}u(x)\geq u(x_0),\quad \forall x_0\in S.
\end{align*}

\begin{defi}
	Let $\Omega$ be an open subset in $\mathbb{R}^2$, we say $u\in USC(\Omega)$ is a viscosity subsolution of 
	\begin{align}\label{viscosity solution}
	\lambda(A^u)\in \partial\Gamma,\quad in \quad \Omega
	\end{align}
	if for any point $x_0\in \Omega$, $\varphi\in C^2(\Omega)$, $(u-\varphi)(x_0)=0$, $u-\varphi\leq 0$ near $x_0$, we have
	\begin{align*}
	\lambda(A^\varphi(x_0))\in \mathbb{R}^2\setminus \Gamma.
	\end{align*}
	
	Similarly, we say $u\in LSC(\Omega)$ is a viscosity supersolution of (\ref{viscosity solution}), if for any point $x_0\in \Omega$, $\varphi\in C^2(\Omega)$, $(u-\varphi)(x_0)=0$, $u-\varphi\geq 0$ near $x_0$, we have
	\begin{align*}
	\lambda(A^\varphi(x_0))\in \bar{\Gamma}.
	\end{align*}
	
	We say $u$ is a viscosity solution of (\ref{viscosity solution}), if it is both a subsolution and a supersolution.
\end{defi}

We remark that our definition is consistent with \cite[Definition 1.1]{L09} and \cite[Definition 1.3]{LNW}.

\subsubsection{Previous results}
\ 

In this section, we state some previous results which we use in this paper. The first one is an asymptotic behaviour for viscosity supersolution.

\begin{prop}\label{Asymptotic behavior-1}(\cite{LLL2})
	Let $\Gamma=\Gamma_p$ for some $1<p\leq 2$, and let $u$ be a viscosity supersolution of (\ref{viscosity solution}) in $\mathbb{R}^2\setminus B_{{r_0}}$ for some $r_0>0$.
	Then there exists $K_0>0$, such that
	\begin{align*}
	\inf_{\partial B_r}u(r)+4\ln r\text{ is monotonically nondecreasing in $r$ for }r>K_0.
	\end{align*}
	Consequently, $\displaystyle\liminf_{x\rightarrow \infty} \left(u(x)+4\ln |x|\right)>-\infty$.
\end{prop}
The above proposition can be equivalently stated as follows:

\begin{manualtheorem}{\ref{Asymptotic behavior-1}'}
	Let $\Gamma=\Gamma_p$ for some $1<p\leq 2$, and let $u$ be a viscosity supersolution of (\ref{viscosity solution}) in $B_{r_0}\backslash\{0\}$ for some $r_0>0$.
	Then there exists $\epsilon>0$, such that
	\begin{align*}
	\inf_{\partial B_r}u(r)\text{ is monotonically nonincreasing in $r$ for }0<r<\epsilon.
	\end{align*}
	Consequently, $\displaystyle\liminf_{x\rightarrow 0} u(x)>-\infty$.
\end{manualtheorem}

\medskip

We now state a lemma concerning comparison principle.  It is a special case of Corollary 1.9 in \cite{LNW}, see also \cite[Proposition 1.14]{L09} for a result of this type in dimensions $n\geq 3$.

\begin{lemm}\label{Comparison-2}
	Let $\Omega$ be an open subset in $\mathbb{R}^2$, $E\subset\Omega$ be a closed set with zero Newtonian capacity. Let $u\in USC(\bar{\Omega})$ be viscosity subsolution of (\ref{viscosity solution}) in $\Omega$ and $v\in LSC(\bar{\Omega}\setminus E)$ be viscosity supersolution of (\ref{viscosity solution}) in $\Omega\setminus E$. Assume further that 
	\begin{align*}
	\inf_{\Omega\setminus E}v>-\infty,
	\end{align*}
	and $u<v$ on $\partial\Omega$, then $\inf_{\Omega\setminus E} (v-u)>0$.
\end{lemm}

\begin{proof}
	Let $w=-u$, we have
	\begin{align*}
	A^u=e^w\left(\nabla^2w+\frac{1}{2}\nabla w\otimes \nabla w-\frac{1}{4} |\nabla w|^2I \right).
	\end{align*}
	Then all the assumptions in Corollary 1.9 in \cite{LNW} are satisfied. Note that we do not need the assumption $u\leq v$ in $\Omega\setminus E$, this can be seen in the proof of Corollary 1.9 in \cite{LNW}.
\end{proof}

\medskip

A consequence of the comparison principle is the Lipschitz regularity for viscosity solution. For related arguments in higher dimensions, see \cite[Lemma 3.1]{LN1} and \cite[Lemma A.2]{LL2}.

\begin{lemm}\label{Lipschitz}
	Let $u\in C^0(B_3)$ be viscosity solution of (\ref{viscosity solution}) in $B_3$. Then $u\in C^{0,1}_{loc}(B_3)$. Furthermore, there exists a constant $C$ such that	
	\begin{align*}
	|\nabla u|\leq C\big(\frac{\sup_{B_{2}}e^u}{\inf_{B_{2}}e^u}\big)^{\frac{1}{4}}\quad\text{in }B_1.
	\end{align*}
\end{lemm}

\begin{proof}

	For $x\in\mathbb{R}^2$, $\lambda>0$, let
	\begin{align}\label{uxlambda}
	u_{x,\lambda}(y):=u(x+\frac{\lambda^2(y-x)}{|y-x|^2})-4\ln\frac{|y-x|}{\lambda},\quad y\in \mathbb{R}^2\backslash\{x\}.
	\end{align}
	
	Let $R>0$ be given by
	\begin{align*}
	4\ln(4R)=\inf_{B_2}u-\sup_{B_2}u.
	\end{align*}
	
	By the above we have
	\begin{align*}
	u_{x,\lambda}(y)\leq \sup_{B_2}u+4\ln 4R=\inf_{B_2}u\leq u(y),\quad\text{for any }|x|\leq 1,\ 0<\lambda\leq R,\ |y|=2,
	\end{align*}
	and $u_{x,\lambda}=u$ on $\partial B(x,\lambda)$.
	
	By the conformal invariance property of $A^u$, if $\lambda(A^u)\in\partial\Gamma$, then $\lambda(A^{u_{x,\lambda}})\in \partial\Gamma$. See our previous paper \cite{LLL2} for details.
	
	Since $\lambda(A^{u_{x,\lambda}})\in \partial\Gamma$ in $B_2\setminus \overline{B(x,\lambda)}$ in the viscosity sense, by Lemma \ref{Comparison-2}, for any $0<\lambda\leq R$, $x\in \overline{B}_1$, we have
	\begin{align*}
	u_{x,\lambda}\leq u,\quad\text{in} \ B_2\setminus \overline{B(x,\lambda)}.
	\end{align*}
	By Lemma \ref{Calculus-2}, $u$ is Lipschitz continuous on $\overline{B}_1$. The lemma is now proved.
	
\end{proof}

\medskip

\subsection{Symmetry and Liouville type theorems for $f(\lambda(A^u))=0$}
\ 

Given the comparison principle and asymptotic behavior, Theorem \ref{dege1} can be proved as in \cite{L07} or \cite{L09}. For reader's convenience, we include the proof below.

\medskip

\begin{proof}
	Let $u_{x,\lambda}$ be defined by (\ref{uxlambda}). For every $x\in\mathbb{R}^2\setminus\{0\}$ and for every $0<\lambda<|x|$, we want to use comparison principle to $u_{x,\lambda}$ and $u$ in $B_\lambda(x)\setminus\{x,\frac{|x|^2-\lambda^2}{|x|^2}x\}$. We only need to check $u_{x,\lambda}$ is bounded below near $x$ and $\frac{|x|^2-\lambda^2}{|x|^2}x$.

	By the asymptotic behavior Proposition \ref{Asymptotic behavior-1},
	\begin{align*}
	\liminf_{x\rightarrow \infty} \left(u(x)+4\ln |x|\right)>-\infty,\quad \liminf_{x\rightarrow 0} (u(\frac{x}{|x|^2})-4\ln |x|)>-\infty.
	\end{align*}
	
	It follows that
	\begin{align*}
	\inf_{B_\lambda(x)\setminus\{x,\frac{|x|^2-\lambda^2}{|x|^2}x\}}u_{x,\lambda}>-\infty.
	\end{align*}
	
	By the comparison principle Lemma \ref{Comparison-2}, we have
	\begin{align}\label{ugequxlambda}
	u(y)\geq u_{x,\lambda}(y),\quad \forall\  0<\lambda<|x|,|y-x|\geq \lambda, y\neq 0.
	\end{align}
	
	For any unit vector $e\in\mathbb{R}^2$, $a>0$, $y\in\mathbb{R}^2$ satisfying $(y-ae)\cdot e<0$, and for any $R>a$, we have, by (\ref{ugequxlambda}) with $x=Re$ and $\lambda=R-a$, 
	\begin{align*}
	u(y)\geq u(x+\frac{\lambda^2(y-x)}{|y-x|^2})-4\ln \frac{|y-x|}{\lambda}.
	\end{align*}
	
	Sending $R$ to infinity, we obtain
	\begin{align*}
	u(y)\geq u(y-2(y\cdot e-a)e),\text{ for any }y\in\mathbb{R}^2\text{ satisfying }(y-ae)\cdot e<0.
	\end{align*}
	
	This gives the radial symmetry of the function $u$ and
	\begin{align*}
	u(y)=u(y_1,y_2)\geq u_a(y):=u(2a-y_1,y_2),\quad \forall\ y_1\leq a,\ a>0.
	\end{align*}
	
	Since $u=u_a$ on $y_1=a$, we have $\frac{\partial (u-u_a)}{\partial y_1}\leq 0$ at $y=(a,0)$, i.e. $u^\prime(a)\leq 0$ whenever $u$ is differentiable. Because  $u$ and $u_a$ satisfy the same equation in $y_1<a$, we have, by Hopf lemma, $\frac{\partial (u-u_a)}{\partial y_1}<0$ at $y=(a,0)$, i.e. $u^\prime(a)< 0$ whenever $u$ is differentiable. Consequently, $u^\prime(r)<0$ a.e. as $u$ is Lipschitz continuous by Lemma \ref{Lipschitz}.
\end{proof}

\medskip

\subsection{Local gradient estimate}
\ 

In this section, we establish Theorem \ref{grad2} and Theorem \ref{grad3}. The proof follows the strategy in \cite{L09}, and relies on the Liouville type theorem Theorem \ref{dege1}.

\medskip

\subsubsection{Gradient estimate assuming lower bound}
\ 

We first give the proof of gradient estimate for the equation $f(\lambda({A^u}))=K$, assuming both upper bound and lower bound of $u$. Namely, we prove the following proposition. The proof is based on Bernstein type arguments.

\begin{prop}\label{grad1}
	Let $\Gamma=\Gamma_p$ for some $1\leq p\leq 2$, $f$ satisfy (\ref{f1})-(\ref{f3}) and (\ref{f4}), and $K$ be a $C^1$ positive function in $B_r\subset\mathbb{R}^2$. For constant $-\infty<\alpha\leq \beta<\infty$, let $u\in C^3(B_{3r})$ satisfy
	\begin{align}\label{eqn2}
	f(\lambda(A^u))=K,\ \alpha\leq u\leq \beta,\ \lambda(A^u)\in \Gamma\quad\text{in }B_{3r}.
	\end{align}
	Then
	\begin{align*}
	|\nabla u|\leq C\quad in\ B_r,
	\end{align*}
	for some constant $C$ depending only on $\alpha,\beta,r$, $\|K\|_{C^1}$ and $(f,\Gamma)$.
\end{prop}

\begin{proof}
	For simplicity, write
	\begin{align*}
	W=A^u=e^{-u}(-\nabla^2 u+\frac{1}{2}du\otimes du-\frac{|\nabla u|^2}{4}\delta_{ij}).
	\end{align*}

	Fix some small constants $\epsilon$,$c_1>0$, depending only on $\alpha$,$\beta$ such that the function $\phi(s):=\epsilon e^s$ satisfies
	\begin{align}\label{propertyphi}
	\phi'\geq c_1,\quad \phi''-\frac{1}{2}\phi'-(\phi')^2\geq 0,\quad on\ [\alpha,\beta].
	\end{align}
	Let $\rho\geq 0$ be a smooth function taking value $1$ in $B_r$ and $0$ outside $B_{2r}$, satisfying $|\nabla \rho|^2\leq C_1\rho$, where $C_1$ depends on $\rho$ only. Consider
	\begin{align*}
	G=\rho e^{\phi(u)}|\nabla u|^2.
	\end{align*}
	Let $G(x_0)=\max_{\bar{B}_{2r}}G$ for some $x_0\in \bar{B}_{2r}$. If $x_0\in\partial B_{2r}$, then $G\equiv 0$. We only need to consider $x_0\in {B}_{2r}$. After a rotation of the axis if necessary, we may assume that $W(x_0)$ is a diagonal matrix. In the following, we use subscripts of a function to denote derivatives. For example, $G_i=\partial_{x_i}G$, $G_{ij}=\partial_{x_ix_j}G$. We also use the notation $f^i:=\frac{\partial f}{\partial \lambda_i}$.
	Applying $\partial_{x_k}$ to equation (\ref{eqn2}), we have
	\begin{align}\label{id1}
	f^iW_{ii,k}=K_k.
	\end{align}
	By calculation,
	\begin{align*}
	G_i=2\rho e^{\phi}u_{ki}u_k+\rho\phi'e^\phi|\nabla u|^2u_i+e^\phi|\nabla u|^2\rho_i=2\rho e^\phi u_{ki}u_k+(\phi'u_i+\frac{\rho_i}{\rho})G.
	\end{align*}
	At $x_0$, we have $G_i=0$, i.e.
	\begin{align}\label{id2}
	2u_{ki}u_k=-\phi'|\nabla u|^2u_i-\frac{\rho_i}{\rho}|\nabla u|^2,\ i=1,2.
	\end{align}
	Notice that $W_{ii}=e^{-u}(-u_{ii}+\frac{1}{2}u_i^2-\frac{1}{4}|\nabla u|^2)$, so we have
	\begin{align}\label{uii}
	u_{ii}=-e^uW_{ii}+\frac{1}{2}u_i^2-\frac{1}{4}|\nabla u|^2.
	\end{align}
	Take second derivatives of $G$ and evaluate at $x_0$,
	\begin{align*}
	0\geq(G_{ij})=&2u_{kij}u_ke^\phi\rho+2u_{ki}u_{kj}e^\phi\rho+2u_{ki}u_ke^\phi\phi'u_j\rho+2u_{ki}u_ke^\phi\rho_j
	\\&+(\phi''u_iu_j+\phi'u_{ij}+\frac{\rho\rho_{ij}-\rho_i\rho_j}{\rho^2})\rho e^\phi|\nabla u|^2+(\phi'u_i+\frac{\rho_i}{\rho})G_j.
	\end{align*}
	The last term above vanishes since $G_j=0$. Therefore, at $x_0$, using formula (\ref{id1}),(\ref{id2}) and (\ref{uii}), we obtain
	\begin{align*}
	0\geq& e^{-\phi}f^iG_{ii}\\
	=&2\rho f^iu_{iik}u_k+2\rho f^i u_{ki}^2+2\rho\phi'f^iu_{ki}u_ku_i+2f^iu_{ki}u_k\rho_i+f^i(\phi''u_i^2+\phi'u_{ii}+\frac{\rho\rho_{ii}-\rho_i^2}{\rho^2})\rho|\nabla u|^2
	\\
	=&2\rho f^iu_k\bigg(-e^uW_{ii}+\frac{1}{2}u_i^2-\frac{1}{4}|\nabla u|^2\bigg)_k+2\rho f^i u_{ki}^2-\rho\phi'f^i\bigg(\phi'|\nabla u|^2u_i^2+\frac{\rho_iu_i}{\rho}|\nabla u|^2\bigg)\\
	&-f^i\rho_i\bigg(\phi'|\nabla u|^2u_i+\frac{\rho_i}{\rho}|\nabla u|^2\bigg)+\rho\phi''|\nabla u|^2f^iu_i^2+\rho\phi'|\nabla u|^2f^i\bigg(-e^uW_{ii}+\frac{1}{2}u_i^2-\frac{1}{4}|\nabla u|^2\bigg)\\
	&+f^i|\nabla u|^2\frac{\rho\rho_{ii}-\rho_i^2}{\rho}\\
	=&2\rho f^i\bigg\lbrace-e^{u}W_{ii,k}u_k-e^{u}|\nabla u|^2W_{ii}-\frac{u_i^2}{2}\phi'|\nabla u|^2-\frac{u_i\rho_i}{2\rho}|\nabla u|^2+\frac{1}{4}\phi'|\nabla u|^4+\sum_j\frac{u_j\rho_j}{4\rho}|\nabla u|^2\bigg\rbrace+2\rho f^i u_{ki}^2\\
	&-\rho\phi'f^i\bigg(\phi'|\nabla u|^2u_i^2+\frac{\rho_iu_i}{\rho}|\nabla u|^2\bigg)-f^i\rho_i\bigg(\phi'|\nabla u|^2u_i+\frac{\rho_i}{\rho}|\nabla u|^2\bigg)\\
	&+\rho\phi''|\nabla u|^2f^iu_i^2+\rho\phi'|\nabla u|^2f^i\bigg(-e^uW_{ii}+\frac{1}{2}u_i^2-\frac{1}{4}|\nabla u|^2\bigg)+f^i|\nabla u|^2\frac{\rho\rho_{ii}-\rho_i^2}{\rho}\\
	=&\bigg\lbrace-2\rho e^uK_ku_k-2\rho e^{u}|\nabla u|^2f- f^iu_i\rho_i|\nabla u|^2+\sum_j\frac{u_j\rho_j}{2}|\nabla u|^2\sum_i f^i\\
	&-2\phi'|\nabla u|^2 f^i\rho_iu_i-\rho e^{u}\phi'|\nabla u|^2f+|\nabla u|^2f^i\frac{\rho\rho_{ii}-2\rho_i^2}{\rho}\bigg\rbrace+2\rho f^i u_{ki}^2\\
	&+\rho\phi'\sum_i f^i\frac{|\nabla u|^4}{4}+(\phi''-(\phi')^2-\frac{1}{2}\phi')\rho f^i|\nabla u|^2u_i^2.
	\end{align*}
	
	In the following, we use $C_2$ to denote some positive constant depending only on $\alpha,\beta,r,K$, and $(f,\Gamma)$ that may vary from line to line. Using (\ref{f3}), (\ref{f4}) and (\ref{propertyphi}), we obtain
	\begin{align*}
	0\geq& e^{-\phi}f^iG_{ii}\\
	\geq& -C_2|\nabla u|\sum_i f^i-C_2|\nabla u|^2\sum_i f^i-C_2\sqrt{\rho}|\nabla u|^3\sum_i f^i+\rho\phi'\sum_i f^i\frac{|\nabla u|^4}{4}\\
	&+(\phi''-(\phi')^2-\frac{1}{2}\phi')\rho f^i|\nabla u|^2u_i^2\\
	\geq& -C_2|\nabla u|\sum_i f^i-C_2|\nabla u|^2\sum_i f^i-C_2\sqrt{\rho}|\nabla u|^3\sum_i f^i+\frac{c_1}{4}\rho\sum_i f^i |\nabla u|^4.
	\end{align*}
	Multiply by $\sqrt{\rho}$, then
	\begin{align*}
	0&\geq -C_2\sqrt{\rho}|\nabla u|\sum_i f^i-C_2|\nabla u|^2\sqrt{\rho}\sum_i f^i-C_2\rho|\nabla u|^3\sum_i f^i+\frac{c_1}{4}\rho^{3/2}\sum_i f^i |\nabla u|^4\\
	&\geq -C_2|\nabla u|\sum_i f^i-C_2|\nabla u|^2\sqrt{\rho}\sum_i f^i-C_2\rho|\nabla u|^3\sum_i f^i+\frac{c_1}{4}\rho^{3/2}\sum_i f^i |\nabla u|^4	\\
	&=|\nabla u|(-C_2-C_2\sqrt{\rho}|\nabla u|-C_2(\sqrt{\rho}|\nabla u|)^2+\frac{c_1}{4}(\sqrt{\rho}|\nabla u|)^3)\sum_i f^i.
	\end{align*}
	Therefore,
	\begin{align*}
	-C_2-C_2\sqrt{\rho}|\nabla u|-C_2(\sqrt{\rho}|\nabla u|)^2+\frac{c_1}{4}(\sqrt{\rho}|\nabla u|)^3\leq 0.
	\end{align*}
	which implies $\rho|\nabla u|^2(x_0)\leq C_2$, so is $G(x_0)$. Since $G(x_0)$ is maximum, $|\nabla u|\leq C_2$ on $B_1$.
\end{proof}

\subsubsection{Gradient estimate for $f(\lambda(A^u))=K$}
\ 

In this section, we prove Theorem \ref{grad2}. Now we no longer assume a lower bound for $u$.

We need to introduce some notations. Let $v$ be a locally Lipschitz function in some open subset $\Omega$ of $\mathbb{R}^2$. For $0<\alpha<1$, $x\in\Omega$, and $0<\delta<dist(x,\partial\Omega)$, let
\begin{align*}
	[v]_{\alpha,\delta}(x):=\sup_{0<|y-x|<\delta}\frac{|v(y)-v(x)|}{|y-x|^\alpha},
\end{align*}
$$ \delta(v,x;\Omega,\alpha):=\left\{
\begin{array}{rcl}
\infty       &      & if\ [v]_{\alpha,dist(x,\partial\Omega)}<1,\\
\mu     &      & where\ 0<\mu\leq dist(x,\partial\Omega)\ and\ \mu^\alpha[v]_{\alpha,\mu}(x)=1\ \\
& &if\ [v]_{\alpha,dist(x,\partial\Omega)}\geq 1
\end{array} \right. $$

Now we prove Theorem \ref{grad2}:

\begin{lemm}\label{holder1}
	Under the assumption of Theorem \ref{grad2}, we have the H\"older estimates:
	\begin{align*}
	\sup_{|y|,|x|<r,|y-x|<2r}\frac{|u(y)-u(x)|}{|y-x|^\alpha}\leq C(\alpha)\quad \forall\ 0<\alpha<1
	\end{align*}
\end{lemm}

\begin{proof}
	Suppose the contrary, then for some $0<\alpha<1$, there exists, in $B_2$, $C^3$ functions $\{u_i\}$, $C^1$ functions $\{K_i\}$ satisfying, for some $\bar{a}>0$,
	\begin{align*}
	\| K_i\|_{C^1(B_2)}&\leq \bar{a}, \quad u_i\leq \bar a\ on\ B_2,\\
	f(\lambda(A^{u_i}))&=K_i,\quad \lambda(A^{u_i})\in\Gamma\quad in\ B_2,
	\end{align*}
	but
	\begin{align*}
	\inf_{x\in B_{1/2}}\delta(u_i,x)\rightarrow 0.
	\end{align*}
	where
	\begin{align*}
	\delta(u_i,x):=\delta(u_i,x;B_2,\alpha)
	\end{align*}
	It follows, for some $x_i\in B_1$,
	\begin{align*}
	\frac{1-|x_i|}{\delta(u_i,x_i)}=\max_{|x|\leq 1}\frac{1-|x|}{\delta(u_i,x)}\rightarrow\infty.
	\end{align*}
	Let
	\begin{align*}
	\sigma_i:=\frac{1-|x_i|}{2},\quad \epsilon_i:=\delta( u_i,x_i).
	\end{align*}
	Then
	\begin{align*}
	\frac{\sigma_i}{\epsilon_i}\rightarrow\infty,\quad\epsilon_i\rightarrow 0.
	\end{align*}
	and
	\begin{align*}
	\epsilon_i\leq 2\delta(u_i,z)\quad\forall\ |z-x_i|<\sigma_i.
	\end{align*}
	Let
	\begin{align}\label{formu2}
	v_i(y):=u_i(x_i+\epsilon_i y)-u_i(x_i),\quad |y|\leq\frac{\sigma_i}{\epsilon_i}.
	\end{align}
	By the definition of $\delta(u_i,x_i)$,
	\begin{align}
	[v_i]_{\alpha,1}(0)=\delta(u_i,x_i)^\alpha[u_i]_{\alpha,\delta(u_i,x_i)}(x_i)=1.\label{formu3}
	\end{align}
	For any $\beta>1$ and $|x|<\beta$, we have that for large $i$,
	\begin{align*}
	|u_i(z)-u_i(x_i+\epsilon_ix)|\leq|u_i(z)-u_i(\frac{1}{2}(z+x_i+\epsilon_ix))|+|u_i(\frac{1}{2}(z+x_i+\epsilon_ix))-u_i(x_i+\epsilon_ix)|,
	\end{align*}
	\begin{align*}
	|z-(x_i+\epsilon_ix)|=2|z-\frac{1}{2}(z+x_i+\epsilon_ix)|=2|\frac{1}{2}(z+x_i+\epsilon_ix)-(x_i+\epsilon_ix)|,
	\end{align*}
	\begin{align*}
	[v_i]_{\alpha,1}(x)&=\epsilon_i^\alpha[u_i]_{\alpha,\epsilon_i}(x_i+\epsilon_ix)\\
	&\leq 2^{-\alpha}\epsilon_i^\alpha(\sup_{|z-(x_i+\epsilon_ix)|<\epsilon_i}[u_i]_{\alpha,\frac{\epsilon_i}{2}}(z)+[u_i]_{\alpha,\frac{\epsilon_i}{2}}(x_i+\epsilon_ix))\\
	&\leq C(\beta)(\sup_{|z-(x_i+\epsilon_ix)|<\epsilon_i}\delta(u_i,z)^\alpha[u_i]_{\alpha,\delta(u_i,z)}(z)+\delta(u_i,x_i+\epsilon_ix)^\alpha[u_i]_{\alpha,\delta(u_i,x_i+\epsilon_ix)}(x_i+\epsilon_ix))\\
	&\leq C(\beta).
	\end{align*}
	This implies for any $\beta>1$,
	\begin{align}\label{formu4}
	-C(\beta)\leq v_i(y)\leq C(\beta),\quad \forall\ |y|\leq \beta.
	\end{align}
	 By Proposition \ref{grad1}, we have, for any $\beta>1$,
	\begin{align*}
	|\nabla v_i(y)|\leq C(\beta),\quad \forall\ |y|<\beta.
	\end{align*}
	Passing to a subsequence,
	\begin{align*}
	v_i\rightarrow v\quad in\ C_{loc}^\gamma(\mathbb{R}^2)\text{ for all }\alpha<\gamma<1,
	\end{align*}
	where $v$ is a function in $C_{loc}^{0,1}(\mathbb{R}^2)$ satisfying $[v]_{\alpha,1}(0)=1$. In particular, $v$ cannot be a constant.
	
	Clearly, for $\gamma_i:=e^{-u_i(x_i)}\epsilon_i^{-2}\rightarrow\infty$, and $x=x_i+\epsilon_iy$,
	\begin{align*}
	f(\gamma_i\lambda(A^{v_i(y)}))=f(\lambda(A^{u_i(x)}))=K_i,\quad |y|<\frac{\sigma_i}{\epsilon_i}.
	\end{align*}
	Thus,
	\begin{align*}
	\lim_{i\rightarrow\infty}f(\lambda(A^{v_i(y)}))=\lim_{i\rightarrow\infty}\gamma_i^{-1}K_i=0.
	\end{align*}
	
	By standard arguments, see e.g. \cite[Theorem 1.10]{L09}, $v$ is a locally Lipschitz viscosity solution of $\lambda(A^v)\in \partial\Gamma$ in $\mathbb{R}^2$.
	
	By Theorem \ref{dege1}, $v$ is a constant. This leads to a contradiction to $[v]_{\alpha,1}(0)=1$. The H\"older estimate is established.	
\end{proof}
Based on the H\"older estimates, we establish Theorem \ref{grad2}:
\begin{proof}
	The H\"older estimate yields the Harnack inequality
	\begin{align*}
	\sup_{B_{2r}}u\leq C+\inf_{B_{2r}}u
	\end{align*}
	Consider $w:=u-u(0)$, the equation of $w$ on $B_{3r}$ is
	\begin{align*}
	f(\lambda(A^w))=e^{u(0)}K,\quad\lambda(A^w)\in\Gamma
	\end{align*}
	and $w$ satisfies
	\begin{align*}
	-C\leq w\leq C\quad \text{in}\ B_{2r}.
	\end{align*}
	Since $u(0)$ is bounded from above, using Proposition \ref{grad1}, we have
	\begin{align*}
	|\nabla u|\leq C\quad \text{in}\ B_r.
	\end{align*}
	Theorem \ref{grad2} is established.
\end{proof}

\subsubsection{Gradient estimate for $f(\lambda(A^u))=0$}
\ 

Now we prove Theorem \ref{grad3}. The proof is similar to the proof of Theorem \ref{grad2}. We will need Theorem \ref{dege1} and the comparison principle Lemma \ref{Comparison-2} to finish the proof.

In the following, for simplicity, write $\delta(v,x,\alpha)=\delta(v,x;B_1,\alpha)$.

Now we give the proof of Theorem \ref{grad3}.

\begin{proof}

Since the equation $\lambda(A^u)\in\partial\Gamma$ is invariant under scaling, it suffices to consider $\epsilon=15/16$. We first claim that
\begin{align*}
\sup_{x\neq y\in B_{1/8}}\frac{|u(x)-u(y)|}{|x-y|^\alpha}\leq C(\Gamma,\alpha)\ for\ any\ 0<\alpha<1
\end{align*}
Assume otherwise the above fails. Then for some $0<\alpha<1$, we can find a sequence of positive $C^{0,1}$ functions $u_i$ in $B_1$ such that $\lambda(A^{u_i})$ but
\begin{align*}
\sup_{x\neq y\in B_{1/8}}\frac{|u_i(x)-u_i(y)|}{|x-y|^\alpha}\rightarrow\infty.
\end{align*}
It follows that for some $x_i\in B_{3/4}$,
\begin{align*}
\frac{3/4-|x_i|}{\delta(u_i,x_i,\alpha)}>\frac{1}{2}\sup_{x\in B_{3/4}}\frac{3/4-|x|}{\delta(u_i,x,\alpha)}\rightarrow\infty.
\end{align*}
Let $\sigma_i=\frac{3/4-|x_i|}{2}$ and $\epsilon_i=\delta(u_i,x_i,\alpha)$. Then
\begin{align*}
\frac{\sigma_i}{\epsilon_i}\rightarrow\infty,\ \epsilon_i\rightarrow 0,\text{and} \ \epsilon_i\leq 4\delta(u_i,z,\alpha)\ \text{for any}\ |z-x_i|\leq\sigma_i.
\end{align*}
We now define $v_i(y)$ as in (\ref{formu2}), then (\ref{formu3}) holds.

For any fixed $\beta>1$ and $|y|<\beta$, there holds, for sufficiently large $i$,
\begin{align*}
[v_i]_{\alpha,1}(y)\leq 4.
\end{align*}
Since $v_i(0)=0$ by definition, we deduce from the above that (\ref{formu4}) holds.

Now we can apply Lemma \ref{Lipschitz} to obtain
\begin{align*}
|\nabla v_i|\leq C(\beta)\quad\text{in }B_{\beta/2}\text{ for all sufficiently large i}.
\end{align*}
Passing to a subsequence, we see that $v_i$ converges in $C^{0,\alpha'}\ (\alpha<\alpha'<1)$ on compact subsets of $\mathbb{R}^2$ to some locally Lipschitz function $v_*$ which satisfies $\lambda(A^{v_*})\in\partial\Gamma$ in the viscosity sense. By Theorem \ref{dege1},
\begin{align*}
v_*\equiv v_*(0)=\lim_{i\rightarrow\infty}v_i(0)=1.
\end{align*}
This contradicts that $[v_i]_{\alpha,1}(0)=1$, in view of convergence of $v_i$ to $v_*$. So we have proved the claim.

Because of the claim, we can find some universal constant $C>0$ such that
\begin{align*}
u(0)-C\leq u\leq u(0)+C\text{ in } B_{1/8}
\end{align*}
Apply Lemma \ref{Lipschitz} again, so we obtain the required gradient estimate in $B_{1/16}$.
\end{proof}

\medskip
\subsection{Local $C^2$ estimate}
\

We prove Theorem \ref{C2} in this part, using Bernstein type arguments.

\begin{proof}
	For convenience, we write $u=-v$. Define $W=\nabla^2 v+\frac{1}{2}dv\otimes dv-\frac{1}{4}|\nabla v|^2\delta_{ij}$. So equation (\ref{feqnR2}) becomes
	\begin{align*}
	f(W)=Ke^{-v}.
	\end{align*}
	It suffices to show $\Delta v$ is bounded. Since $\lambda(A^u)\in\Gamma\subset\Gamma_1$, we know $0<tr(W)=\Delta v$.
	
	Without loss of generality, we may assume $r=1$. Let $Q=\eta(\Delta v+\frac{1}{2}|\nabla v|^2)=\eta H$, where $\eta$ is a cut-off function satisfying
	\begin{align*}
	&0\leq\eta\leq 1,\\
	&\eta=1\text{ in }B_{1/2}\text{ and }\eta=0\text{ outside }B_1,\\
	&|\nabla \eta|< C\sqrt{\eta},\\
	&|\nabla^2\eta|<C.
	\end{align*}
	Now we need to get the upper bound for $H$.
	
	Suppose $x_0$ is the maximal point of $Q$. At $x_0$, we have
	\begin{align}\label{H1}
	0=Q_i=\eta_iH+\eta H_i=\eta_i(\Delta v+\frac{1}{2}|\nabla v|^2)+\eta(v_{kki}+v_kv_{ki}),
	\end{align}
	and
	\begin{align*}
	Q_{ij}=\eta_{ij}H+\eta_iH_j+\eta_jH_i+\eta H_{ij}=(\eta_{ij}-2\frac{\eta_i\eta_j}{\eta})H+\eta H_{ij}.
	\end{align*}
	Here we have used (\ref{H1}). We know $Q_{ij}$ is negative semidefinite.
	\begin{align}\label{Hij}
	H_{ij}=v_{kkij}+v_{ki}v_{kj}+v_{k}v_{kij}.
	\end{align}
	Now by the condition $f_{\lambda_i}>0$, $f^{ij}=\frac{\partial f}{\partial W_{ij}}$ is positive definite. So use the condition on $\eta$, we have
	\begin{align}\label{H2}
	0\geq f^{ij}Q_{ij}=f^{ij}((\eta_{ij}-2\frac{\eta_i\eta_j}{\eta})H+\eta H_{ij})\geq  -C\sum_i f^{ii}H+\eta f^{ij}H_{ij}.
	\end{align}
	Using (\ref{Hij}), we obtain
	\begin{align*}
	f^{ij}H_{ij}&=f^{ij}(v_{kkij}+v_{ki}v_{kj}+v_{k}v_{kij})=I+II,
	\end{align*}
	where $I=f^{ij}v_{ijkk}$ and $II=f^{ij}(v_{ki}v_{kj}+v_{k}v_{ijk})$.
	
	To compute I, notice that
	\begin{align*}
	W_{ij,kk}=v_{ijkk}+\frac{1}{2}(v_{ikk}v_j+2v_{ik}v_{jk}+v_iv_{jkk})-\frac{1}{2}(|\nabla^2 v|^2+v_lv_{lkk})\delta_{ij}.
	\end{align*}
	Then
	\begin{align*}
	I&=f^{ij}(W_{ij,kk}-\frac{1}{2}(v_{ikk}v_j+2v_{ik}v_{jk}+v_iv_{jkk})+\frac{1}{2}(|\nabla^2 v|^2+v_lv_{lkk})\delta_{ij})\\
	&= f^{ij}W_{ij,kk}+f^{ij}(-(v_{ikk}v_j+v_{ik}v_{jk})+\frac{1}{2}(|\nabla^2 v|^2+v_lv_{lkk})\delta_{ij}).
	\end{align*}
	Now use (\ref{H1}) to replace $v_{ikk}$ and $v_{lkk}$,
	\begin{align*}
	I&= f^{ij}W_{ij,kk}+f^{ij}((\frac{\eta_i}{\eta}Q+v_kv_{ki})v_j-v_{ik}v_{jk}+\frac{1}{2}|\nabla^2 v|^2\delta_{ij}-\frac{1}{2}v_l(\frac{\eta_l}{\eta}Q+v_kv_{kl})\delta_{ij})\\
	&=f^{ij}W_{ij,kk}+f^{ij}(v_kv_{ki}v_j-v_{ik}v_{jk}+\frac{1}{2}|\nabla^2 v|^2\delta_{ij}-\frac{1}{2}v_lv_kv_{kl}\delta_{ij})+f^{ij}Q(\frac{\eta_i}{\eta}v_j-\frac{1}{2}v_l\frac{\eta_l}{\eta}\delta_{ij}).
	\end{align*}
	Using the condition on $\eta$, we obtain
	\begin{align*}
	I\geq f^{ij}W_{ij,kk}+f^{ij}(v_kv_{ki}v_j-v_{ik}v_{jk}+\frac{1}{2}|\nabla^2 v|^2\delta_{ij}-\frac{1}{2}v_lv_kv_{kl}\delta_{ij})-C\sum_i f^{ii}H\eta^{-1/2}.
	\end{align*}
	For II, we use the formula
	\begin{align*}
	W_{ij,k}=v_{ijk}+\frac{1}{2}v_{ik}v_j+\frac{1}{2}v_{jk}v_i-\frac{1}{2}v_lv_{lk}\delta_{ij}
	\end{align*}
	to replace $v_{ijk}$, then we obtain
	\begin{align*}
	II&=f^{ij}(v_{ki}v_{kj}+v_{k}v_{kij})\\
	&= v_kf^{ij}W_{ij,k}+f^{ij}(v_{ki}v_{kj}-v_kv_{ik}v_j+\frac{1}{2}v_kv_lv_{lk}\delta_{ij}).
	\end{align*}
	Combine I and II, then
	\begin{align*}
	f^{ij}H_{ij}\geq &f^{ij}W_{ij,kk}+v_kf^{ij}W_{ij,k}+f^{ij}(v_kv_{ki}v_j-v_{ik}v_{jk}+\frac{1}{2}|\nabla^2 v|^2\delta_{ij}-\frac{1}{2}v_lv_kv_{kl}\delta_{ij})\\&+f^{ij}(v_{ki}v_{kj}-v_kv_{ik}v_j+\frac{1}{2}v_kv_lv_{lk}\delta_{ij})-C\sum_i f^{ii}H\eta^{-1/2}\\
	=&f^{ij}W_{ij,kk}+v_kf^{ij}W_{ij,k}+\frac{1}{2}f^{ij}|\nabla^2 v|^2\delta_{ij}-C\sum_i f^{ii}H\eta^{-1/2}.
	\end{align*}
	Now multiply by $\eta$ on (\ref{H2}). In the following, we use $C_2$ to denote some positive constant depending only on $C$ and $|\nabla u|$, that may vary from line to line.
	\begin{align*}
	0&\geq -C\eta\sum_i f^{ii}H+\eta^2 f^{ij}H_{ij}\\
	&\geq  \eta^2f^{ij}W_{ij,kk}+\eta^2v_kf^{ij}W_{ij,k}+\frac{1}{2}\eta^2\sum_i f^{ii}|\nabla^2 v|^2-C\sum_i f^{ii}(H\eta^{3/2}+\eta H)\\
	&\geq \eta^2f^{ij}W_{ij,kk}+\eta^2v_kf^{ij}W_{ij,k}+\frac{1}{2}\eta^2\sum_i f^{ii}|\nabla^2 v|^2-C_2\sum_i f^{ii}(1+\eta|\nabla^2 v|).
	\end{align*}
	By concavity of $f$, $(Ke^{-v})_{kk}=f_{kk}\leq f^{ij}W_{ij,kk}$. Use the property that $\sum f_{\lambda_i}\geq \delta$, we obtain
	\begin{align*}
	0&\geq  \eta^2(Ke^{-v})_{kk}+\eta^2v_k(Ke^{-v})_k+\frac{1}{2}\eta^2\sum_i f^{ii}|\nabla^2 v|^2-C_2\sum_i f^{ii}(1+\eta|\nabla^2 v|)\\
	&\geq \frac{1}{2}\eta^2\sum_i f^{ii}|\nabla^2 v|^2-C_2\sum_i f^{ii}(1+\eta|\nabla^2 v|).
	\end{align*}
	So
	\begin{align*}
	0\geq\eta^2|\nabla^2 v|^2-C_2(1+\eta|\nabla^2 v|).
	\end{align*}
	This implies $\eta|\nabla^2 v|(x_0)\leq C_2$, so $\max_{B_1}Q(x)=Q(x_0)=\eta (\Delta v+\frac{1}{2}|\nabla v|^2)(x_0)\leq C_2$. Therefore, $\Delta v\leq C_2$ in $B_{1/2}$. Using gradient estimate Theorem \ref{grad2}, we reach the conclusion.
\end{proof}

\medskip

\section{B\^ocher type theorems}\label{bochertype}

In this section, we prove B\^ocher type theorems Theorem \ref{Bocher1} and \ref{Bocher2}. Our proof uses the idea in \cite{LN1}, with some twists.

\subsection{Classification of Radially Symmetric Case}
\ 

We will classify all the radially symmetric viscosity solutions of $\lambda(A^u)\in\partial\Gamma$ in annulus domain.

Now we state this classification theorem.
\begin{theo}\label{radialsol}
	Let $\Gamma=\Gamma_p$ for some $1\leq p\leq 2$. Denote $r=|x|$. Then all radially symmetric (continuous) viscosity solutions of $\lambda(A^u)\in\partial\Gamma$ in annulus $\{a<|x|<b\}$, $a\geq 0$ and $b\leq \infty$, are classified as follows:
	
	(a) $u=C_1 \ln r+C_2$, $C_1,C_2\in\mathbb{R}$, if $p=1$,
	
	(b) $u=-4\ln r+C_1$ or $u=C_1$, $C_1\in\mathbb{R}$, if $1<p\leq 2$,
	
	(c) $u=-\dfrac{4}{p-1}\ln(r^{p-1}+C_1)+C_2$, where $C_1>0$, if $1< p\leq 2$,
	
	(d) $u=\dfrac{4(2-p)}{p-1}\ln(r^{-(p-1)/(2-p)}-C_1)+C_2$, where $C_1>0$, if $1<p<2$, $0\leq b\leq C_1^{-(2-p)/(p-1)}$,
	
	(e) $u=\dfrac{4(2-p)}{p-1}\ln(C_1-r^{-(p-1)/(2-p)})+C_2$, where  $C_1>0$, if $1<p<2$, $C_1^{-(2-p)/(p-1)}\leq a\leq \infty$.
\end{theo}
\begin{rema}
	When $n\geq 3$, any radially symmetric (continuous) viscosity solution of $\lambda(A^u)\in\partial\Gamma_k$ for $2\leq k\leq n$ in an annulus $\{a<|x|<b\}$ can always be extended to a solution in $\mathbb{R}^n\backslash\{0\}$. However, this does not hold for $n=2$, as shown in Theorem \ref{radialsol} (d)(e).
\end{rema}

\begin{proof}
	First we assume that $u$ is smooth.
	
	Let $u=u(r)$ satisfy the equation $\lambda(A^u)\in\partial\Gamma$. Denote $\lambda(A^u)=(\lambda_1,\lambda_2)$, then a calculation gives:
	\begin{align*}
	\lambda_1&:=\frac{1}{4e^u}((u')^2-4u''),\\
	\lambda_2&:=\frac{1}{4e^u}(-\frac{1}{r}u'(4+ru')).
	\end{align*}
	
	Take a point $P\in(a,b)$. We have three cases: $\lambda_2(P)>0$, $\lambda_2(P)<0$, $\lambda_2(P)=0$.
	
	\textbf{ In case 1}, $\lambda_2(P)>0$, assume there is a maximal interval $(c,d)\subset(a,b)$ containing $P$ such that $\lambda_2>0$ in $(c,d)$. Since $\lambda(A^u)\in\partial\Gamma$,  $\lambda_1+(2-p)\lambda_2=0$ in $(c,d)$. So we have in $(c,d)$,
	
	\begin{equation*}  
	\left\{  
	\begin{array}{lr}
	-\dfrac{1}{r}u'(4+ru')>0,\\  
	p-1-\dfrac{4u''}{(u')^2}-\dfrac{2-p}{r}\dfrac{4}{u'}=0.\\
	\end{array}
	\right.
	\end{equation*} 
	The first equation implies $u'<0$.
	
	If $p=1$, $u''+\dfrac{1}{r}u'=0$, then $u=C_1\ln r+C_2$ in $(c,d)$. The condition $\lambda_2>0$ implies that $-4<C_1<0$.
	
	If $p\neq 1$, then let $g=\dfrac{4}{u'}$, we obtain $p-1+g'-\dfrac{(2-p)g}{r}=0$. So we can solve this equation and obtain $g=-r-C_3 r^{2-p}$.
	Hence
	\begin{align*}
	u=-\dfrac{4}{p-1}\ln(r^{p-1}+C_3)+C_4.
	\end{align*}
	
	Since $\lambda_2>0$, $C_3>0$.
	
	Now we prove that $(a,b)=(c,d)$. In fact, from the explicit form of $u$, we can see $u'\neq 0$ in $[c,d]$. If $a\neq c$, for example, then $\lambda_2(c)=0$, i.e.
	\begin{align}\label{formu1}
	4+cu'(c)=0.
	\end{align}
	
	If $p=1$, (\ref{formu1}) implies $C_1=-4$, contradiction;
	
	If $p\neq 1$, (\ref{formu1}) implies $C_3=0$, then $u=-4\ln r+C_4$, by direct computation, $\lambda_2$ is identically 0 in $(c,d)$. This is a contradiction.
	
	Thus, $(a,b)=(c,d)$. 
	
	\textbf{In case 2}, $\lambda_2(P)<0$, consider the maximal interval $(c,d)\subset(a,b)$ containing $P$ such that $\lambda_2<0$ in $(c,d)$. Since $\lambda(A^u)\in\partial\Gamma$ in $(c,d)$,  $(2-p)\lambda_1+\lambda_2=0$. Obviously $p\neq 2$. So we have in $(c,d)$,
	\begin{equation*}  
	\left\{  
	\begin{array}{lr}
	-\dfrac{1}{r}u'(4+ru')<0,\\  
	(1-p)-4(2-p)\dfrac{u''}{(u')^2}-\dfrac{1}{r}\dfrac{4}{u'}=0.\\
	\end{array}
	\right.
	\end{equation*}
	
	If $p=1$, then $u=C_1\ln r+C_2$. It follows from $\lambda_2<0$ that $C_1(4+C_1)>0$. So $C_1>0$ or $C_1<-4$.
	
	If $p\neq 1$, then
	\begin{align*}
	u=\dfrac{4(2-p)}{p-1}\ln|r^{-(p-1)/(2-p)}-C_5|+C_6.
	\end{align*}
	
	Since $\lambda_2<0$, we obtain $C_5>0$.
	
	Now we prove that $(c,d)=(a,b)$. Otherwise, we will have $\lambda_2=0$ on the boundary of $(c,d)$, which implies:
	
	If $p=1$, $u\equiv constant$, and $\lambda_2\equiv 0$;
	
	If $p\neq 1$, $C_5=0$, so $u=-4\ln r+C_6$, and $\lambda_2\equiv 0$.
	
	Both lead to a contradiction.

	\textbf{In case 3}, $\lambda_2(P)=0$, by previous arguments, we must have $\lambda_2=0$ on $(a,b)$. it is easy to see either $u\equiv Const$ or $u=C-4\ln r$. Then we can compute easily that $\lambda_1=0$. 
	
	We have proved that all smooth solutions $u$ must be one of (a)-(e). It is straightforward to check that function given by (a)-(e) satisfies the equation $\lambda(A^u)\in\partial\Gamma$.
	\medskip
	
	Now we consider viscosity solutions $u$.
	
		\begin{lemm}\label{radbehavior}
		Assume $\Gamma=\Gamma_2$. For $0\leq a<b\leq\infty$, let $u\in LSC(\{a<|x|<b\})$ be radially symmetric and satisfy $\lambda(A^u)\in\bar{\Gamma}$ in $\{a<|x|<b\}$ in the viscosity sense. Then $u$ is non-increasing and $u(x)+4\ln|x|$ is non-decreasing in $|x|$, i.e. for $a<c<d<b$,
		\begin{align*}
		0\leq u(c)-u(d)\leq 4 (\ln d-\ln c).
		\end{align*}
		In particular, $u$ is locally Lipschitz in $\{a<|x|<b\}$.
	\end{lemm}
	
	\begin{proof}
		If $u$ is in $C^2$, then
		\begin{align*}
		\lambda_2=\frac{1}{4e^u}(-\frac{1}{r}u'(4+ru'))\geq 0,
		\end{align*}
		which implies $u'\leq 0$ and $(u+4\ln r)'\geq 0$. The lemma is already proved in this case.
		
		Now consider $u\in LSC$. Define $w=e^{-u/2}$, then $A^u=2w B_w$, where
		\begin{align*}
		B_w:=\nabla^2 w-\frac{1}{2w}|\nabla w|^2I.
		\end{align*}
		
		Let $\rho\in C_c^{\infty}(\mathbb{R}^2)$ supported in $B_1$ satisfying $\rho\geq 0$ and $\displaystyle\int_{\mathbb{R}^2}\rho=1$. For $\epsilon>0$, let $\rho_\epsilon(x)=\dfrac{1}{\epsilon^2}\rho(\dfrac{x}{\epsilon})$ and let $w_\epsilon:=w*\rho_\epsilon$ in $\{a+\epsilon<|x|<b-\epsilon\}$. Set $u_\epsilon=-2\ln w_\epsilon$. We know $u_\epsilon\rightarrow u$ a.e. as $\epsilon\rightarrow 0^+$.
		
		Since $\lambda(A^u)\in\bar\Gamma$, we know $\lambda(B_w)\in\bar\Gamma$. By the convexity of $B_w$ in $w$ as pointed out in \cite[Lemma A.1]{LNW1}, $\lambda(B_{w_\epsilon})\in\bar\Gamma$, which implies $\lambda(A^{u_\epsilon})\in\bar\Gamma$. Hence $u_\epsilon'\leq 0$ and $(u_\epsilon+4\ln r)'\geq 0$. It follows that, for $\forall$ $a<c<d<b$,
		\begin{align*}
		0\leq u_\epsilon(c)-u_\epsilon(d)\leq 4 (\ln d-\ln c).
		\end{align*}
		
		Sending $\epsilon\rightarrow 0^+$, we obtain
		\begin{align*}
		0\leq u(c)-u(d)\leq 4 (\ln d-\ln c).
		\end{align*}
	\end{proof}

	The next two corollaries concern the existence and uniqueness of radially symmetric viscosity solutions on any finite annulus with given boundary values, as well as their regularity.
	\begin{coro}\label{gamma2rigid}
		Assume $\Gamma=\Gamma_2$. For any $0<a<b<\infty$, $\alpha,\beta$, there exists a radially symmetric function $u$ in $C^0(\{a\leq|x|\leq b\})$ satisfying
		$$ \left\{
		\begin{array}{rcl}
		&\lambda(A^u)\in\partial\Gamma\quad\text{in}&\{a<|x|<b\},\\
		&u\big|_{\partial B_a}=\alpha,&u\big|_{\partial B_b}=\beta,
		\end{array} \right. $$
		if and only if
		\begin{align}\label{condisol}
		0\leq \alpha-\beta\leq 4\ln\frac{b}{a}.
		\end{align}
		Moreover, such solution is unique, and $u\in C^\infty(\{a\leq |x|\leq b\})$.
	\end{coro}
	
	\begin{proof}
		\ 
		
		By Lemma \ref{radbehavior}, (\ref{condisol}) is necessary for solvability. The uniqueness follows from the comparison principle Lemma \ref{Comparison-2}.
		Now we prove the existence part.  Without loss of generality, we can assume that $\alpha=0$. If $\beta=0$ then $u\equiv 0$ is the solution. Now assume $\beta\neq 0$. Then $\dfrac{a}{b}\leq e^{\beta/4}< 1$. It is easy to check that the solution is given by $u=-4\ln(r+C_1)+C_2$, where
		\begin{align*}
		C_1=\frac{a-be^{\beta/4}}{e^{\beta/4}-1}\geq 0,\ C_2=4\ln\frac{a-b}{e^{\beta/4}-1}+\beta.
		\end{align*}
		Clearly, $u\in C^\infty$.
	\end{proof}
	
	\begin{coro}\label{nongamma2}
		Assume $\Gamma=\Gamma_p$, where $1\leq p<2$. For any $0<a<b<\infty$, $\alpha,\beta\in\mathbb{R}$, there exists a unique radially symmetric function $u$ in $C^0(\{a\leq |x|\leq b\})$ satisfying
		$$ \left\{
		\begin{array}{rcl}
		&\lambda(A^u)\in\partial\Gamma\quad\text{in}&\{a<|x|<b\},\\
		&u\big|_{\partial B_a}=\alpha,&u\big|_{\partial B_b}=\beta.
		\end{array} \right. $$
		Moreover, $u\in C^\infty(\{a\leq |x|\leq b\})$.
	\end{coro}
	
	\begin{proof}
		We only need to prove the existence, the remaining is clear. When $p=1$, the existence is obvious since $u$ can be taken of the form $C_1\ln r+C_2$. So we only consider $1<p<2$. Without loss of generality, we assume $\alpha=0$.
		
		If $\beta=0$, then $u\equiv 0$ is a solution.
		
		If $\beta<-4\ln (b/a)$, take $u=\dfrac{4(2-p)}{p-1}\ln(r^{-(p-1)/(2-p)}-C_1)+C_2$ with
		\begin{align*}
		C_1=\frac{b^{-(p-1)/(2-p)}-a^{-(p-1)/(2-p)}e^{(\beta/4)*(p-1)/(2-p)}}{1-e^{(\beta/4)*(p-1)/(2-p)}},
		\end{align*}
		then $u$ satisfies the equation in the annulus. One can see that there is a unique $C_2$ so that $u$ satisfies the boundary conditions. By a direct computation, we can check that $C_1>0$ and $b\leq C_1^{-(2-p)/(p-1)}$.
		
		If $-4\ln (b/a)\leq\beta<0$, take $u=-\dfrac{4}{p-1}\ln(r^{p-1}+C_1)+C_2$ with
		\begin{align*}
		C_1=\frac{a^{p-1}-e^{\beta(p-1)/4}b^{p-1}}{e^{\beta(p-1)/4}-1}\geq 0,
		\end{align*}
		then $u$ satisfies the equation in the annulus. One can see that there is a unique $C_2$ so that $u$ satisfies the boundary conditions.
		
		If $\beta>0$, take $u=\dfrac{4(2-p)}{p-1}\ln(C_1-r^{-(p-1)/(2-p)})+C_2$ with
		\begin{align*}
		C_1=\frac{a^{-(p-1)/(2-p)}e^{(\beta/4)*(p-1)/(2-p)}-b^{-(p-1)/(2-p)}}{e^{(\beta/4)*(p-1)/(2-p)}-1},
		\end{align*}
		then $u$ satisfies the equation in the annulus. One can see that there is a unique $C_2$ so that $u$ satisfies the boundary conditions. By a direct computation, we can check that $C_1>0$ and $C_1^{-(2-p)/(p-1)}\leq a$.			
	\end{proof}

	By now, we have proved that every radially symmetric viscosity solution of $\lambda(A^u)\in\partial\Gamma$ is a smooth solution. Therefore, we have completed the proof of Theorem \ref{radialsol}.
	
\end{proof}

\subsection{A comparison type result}
\

The next lemma shows that the strong maximum principle holds for radially symmetric viscosity solutions of $\lambda(A^u)\in\partial\Gamma$.
\begin{lemm}\label{comp2}
	Let $\Gamma=\Gamma_p$, where $1<p\leq 2$. For $0\leq a<b\leq\infty$, let $u\in C^0(\{a<|x|<b\})$ and $\bar{u}\in LSC(\{a<|x|<b\})$ be radially symmetric and satisfy respectively $\lambda(A^u)\in\partial\Gamma$ and $\lambda(A^{\bar{u}})\in\bar\Gamma$ in $\{a<|x|<b\}$ in the viscosity sense. Assume that $u\leq \bar{u}$ in $\{a<|x|<b\}$. Then
	\begin{align*}
	\text{either }u<\bar{u}\text{ in }\{a<|x|<b\}\text{ or }u\equiv \bar{u} \text{ in }\{a<|x|<b\}
	\end{align*}
\end{lemm}

\begin{proof}
	Suppose for some $c,d\in(a,b)$, $u(c)<\bar{u}(c)$ and $u(d)=\bar{u}(d)$. We may assume that $c<d$; the other case can be proved similarly. By Theorem \ref{radialsol}, $u$ is smooth and takes some specific form.
	
	We first observe that
	\begin{align*}
	u\equiv\bar{u}\text{ in }\{d\leq|x|<b\}.
	\end{align*}
	In fact, if there exists some $d<\bar r<b$ such that $u(\bar r)<\bar u(\bar r)$, then by comparison principle on $\{c<|x|<\bar r\}$, we obtain, for small $\epsilon>0$, $(1+\epsilon)u\leq \bar u$ in $\{c<|x|<\bar r\}$, violating $u(d)=\bar u(d)$.
	
	Fix a $\bar{d}\in(d,b)$, and let $\alpha=\frac{1}{2}(u(c)+\bar{u}(c))$.
	
	Case (i):If $\Gamma=\Gamma_2$, then apply Lemma \ref{radbehavior} to both $u$ and $\bar u$ to get
	\begin{align*}
	0\leq u(c)-u(\bar d)<\alpha-u(\bar d)<\bar u(c)-\bar{u}(\bar d)\leq 4\ln\frac{\bar d}{c}
	\end{align*}
	By Corollary \ref{gamma2rigid}, there exists a unique $C^2$ radially symmetric solution $v$ of $\lambda(A^u)\in\partial \Gamma$ in $\{c<|x|<\bar d\}$ satisfying $v(c)=\alpha$ and $v(\bar{d})=u(\bar{d})=\bar{u}(\bar{d})$.
	
	Case (ii): If $\Gamma\neq\Gamma_2$, then the existence of $v$ is given by Corollary \ref{nongamma2}.
	
	By comparison principle, $v\leq \bar{u}$ on $\{c<|x|<d\}$. On the other hand, since $u(c)<v(c)$ and $u(\bar{d})=v(\bar{d})$, we have, in view of the explicit form of radial solutions given by Theorem \ref{radialsol}, $u<v$ in $\{c<|x|<\bar d\}$. Thus, $u(d)<v(d)\leq\bar{u}(d)$, contradiction.
\end{proof}

A consequence of the above lemma is the following comparison type result.

\begin{coro}\label{comp3}
	Let $\Gamma=\Gamma_p$, where $1<p\leq 2$. For $0\leq a<b<\infty$, let $u\in C^0(\{a\leq|x|\leq b\})$, $\bar{u}\in LSC(\{a\leq|x|\leq b\})$ be radially symmetric and satisfy respectively $\lambda(A^u)\in\partial\Gamma$ and $\lambda(A^{\bar u})\in\bar\Gamma$ in $\{a<|x|<b\}$ in the viscosity sense. Assume that $u(b)\leq\bar{u}(b)$ and $u(d)\geq \bar{u}(d)$ for some $a<d<b$. Then
	\begin{align*}
	\bar{u}\leq u\text{ in }\{a<|x|<d\}
	\end{align*}
	Moreover, if $u(b)<\bar u(b)$, then $\bar u<u$ in $(a,d)$.
\end{coro}

\begin{proof}
	Assume the contrary that $u(c)<\bar u(c)$ for some $c\in(a,d)$. According to Theorem \ref{radialsol}, $u$ is a smooth function. We also know that $u(b)\leq \bar{u}(b)$. An application of comparison principle yields that
	\begin{align*}
	\bar{u}\geq u\text{ on }\bar{B}_b\backslash B_c
	\end{align*}
	In particular, $\bar{u}(d)\geq u(d)$. From the boundary condition, it follows that $\bar{u}(d)=u(d)$. Now by Lemma \ref{comp2}, we obtain $u\equiv\bar u$ on $\bar{B}_b\backslash B_c$, violating $u(c)<\bar u(c)$.
\end{proof}

\subsection{Removable singularity result}
\ 

We will need the following removable singularity result.
\begin{lemm}\label{removable}
	Let $\Gamma=\Gamma_p$, where $1<p\leq 2$. Let $u\in LSC(B_1\backslash\{0\})$ be a viscosity solution of $\lambda(A^u)\in\bar\Gamma$ in $B_1\backslash\{0\}$. Then $u$, with $u(0)=\liminf_{x\rightarrow 0}u(x)$, is a function in $LSC(B_1)$ satisfying $\lambda(A^u)\in\bar\Gamma$ in $B_1$ in the viscosity sense.
\end{lemm}
\begin{proof}
	It is easy to see that $u$, with $u(0)=\liminf_{x\rightarrow 0}u(x)$ is in $LSC(B_1)$. This lemma is a special case of Theorem 3.7 in \cite{LNW}.
\end{proof}

\begin{prop}\label{removable2}
	Let $\Gamma=\Gamma_p$, where $1<p\leq 2$. Let $u\in LSC(B_1\backslash\{0\})\cap L_{loc}^\infty(B_1\backslash\{0\})$ be a function satisfying $\lambda(A^u)\in\bar\Gamma$ in $B_1\backslash\{0\}$ in the viscosity sense and
	\begin{align*}
	\liminf_{|x|\rightarrow 0}(u(x)+4\ln |x|)=-\infty.
	\end{align*}
	Then the function $u$ with $u(0)=\liminf_{|x|\rightarrow 0}u(x)$ is in $C_{loc}^{0,p-1}(B_1)$. Moreover, set $w=\exp(-\frac{p-1}{4}u)$, then
	\begin{align*}
	\|w\|_{C^{0,p-1}(B_{1/2})}\leq C(\Gamma)\max_{\partial B_{3/4}}w.
	\end{align*}
\end{prop}

\begin{proof}
	By Lemma \ref{removable}, $\lambda(A^u)\in\bar\Gamma$ in viscosity sense. Let $v(x)=v(|x|)=\min_{\partial B_{|x|}}u$. Then $\lambda(A^v)\in\bar\Gamma$ in $B_1$ in the viscosity sense, hence $v$ is superharmonic. It follows that $v$ is nonincreasing.
	
	By the hypothesis, $\liminf_{r\rightarrow 0}(v(r)+4\ln r)=-\infty$, hence there exists $0<r_1<3/4$ such that
	\begin{align*}
	v(r_1)+4\ln r_1<v(\frac{3}{4})+4\ln\frac{3}{4}.
	\end{align*}
	Thus,  there exists $C_1>0$ and $C_2$ such that the function
	\begin{align*}
	\hat{v}(r)=-\frac{4}{p-1}\ln(r^{p-1}+C_1)+C_2
	\end{align*}
	satisfies $\hat{v}(r_1)=v(r_1)$ and $\hat{v}(3/4)=v(3/4)$. By Theorem \ref{radialsol}, $\lambda(A^{\hat{v}})\in\partial\Gamma$ in $B_1\backslash\{0\}$. By Corollary \ref{comp3}, we have $v\leq \hat{v}$ in $(0,r_1)$. In particular, $v$ is bounded from above at the origin, and
	\begin{align*}
	u(0)=\liminf_{|x|\rightarrow 0}u=\liminf_{r\rightarrow 0}v(r)<\infty.
	\end{align*}
	By Lemma \ref{removable}, $\lambda(A^u)\in\bar\Gamma$ in the viscosity sense. By the superharmonicity of u,
	\begin{align*}
	c:=\inf_{B_{3/4}}u=\min_{\partial B_{3/4}}u>-\infty.
	\end{align*}
	For $\bar{x}\in B_{1/2}$, consider
	\begin{align*}
	\xi_{\bar{x}}(x):=-\frac{4}{p-1}\ln(4^{p-1}|x-\bar{x}|^{p-1}+b)+c,
	\end{align*}
	where $b>0$ satisfies
	\begin{align}\label{formulab}
	\xi_{\bar{x}}(\bar{x})=-\frac{4}{p-1}\ln b+c=u(\bar{x}).
	\end{align}
	We will show that
	\begin{align}\label{ugeqksi}
	u\geq\xi_{\bar{x}} \quad\text{in }B_{3/4}.
	\end{align}
	It is easy to see that
	\begin{align*}
	\xi_{\bar{x}}(x)\leq -4\ln(4|x-\bar{x}|)+c\leq c \text{ for all }x\in\partial B_{3/4}.
	\end{align*}	
	Also, since $\xi_{\bar{x}}(\bar{x})=u(\bar{x})$, for any $\epsilon>0$, there exists $0<\delta<\frac{1}{8}$ such that
	\begin{align*}
	\xi_{\bar{x}}(x)-\epsilon\leq u\text{ in }B_\delta(\bar{x}).
	\end{align*}
	Since by Theorem \ref{radialsol},
	\begin{align*}
	\lambda(A^{\xi_{\bar{x}}-\epsilon})\in\partial\Gamma\text{ in }B_{3/4}\backslash\{\bar{x}\},
	\end{align*}
	and
	\begin{align*}
	\xi_{\bar{x}}-\epsilon\leq u\text{ on }\partial(B_{3/4}\backslash B_\delta(\bar{x})),
	\end{align*}
	we can apply comparison principle to get
	\begin{align*}
	\xi_{\bar{x}}-\epsilon\leq u\text{ in }B_{3/4}\backslash B_\delta(\bar{x}).
	\end{align*}
	Therefore,
	\begin{align*}
	\xi_{\bar{x}}-\epsilon\leq u\text{ in }B_{3/4}.
	\end{align*}
	Sending $\epsilon\rightarrow 0$, we obtain $u\geq\xi_{\bar{x}}$ in $B_{3/4}$.
	
	Now set
	\begin{align*}
	w(x)=\exp(-\frac{p-1}{4}u(x)),
	\end{align*}
	then it follows that, using (\ref{formulab}) and (\ref{ugeqksi}),
	\begin{align*}
	w(x)-w(\bar{x})\leq 4^{p-1}{|x-\bar{x}|^{p-1}}\max_{\partial B_{3/4}}w\text{ for all }x,\bar{x}\in B_{1/2}.
	\end{align*}
	Switching role of $x$ and $\bar{x}$ we obtain
	\begin{align*}
	|w(x)-w(\bar{x})|\leq 4^{p-1}{|x-\bar{x}|^{p-1}}\max_{\partial B_{3/4}}w\text{ for all }x,\bar{x}\in B_{1/2}.
	\end{align*}
	which proves the result.
\end{proof}

\subsection{Proof of B\^ocher type theorems}
\

\subsubsection{The case $\Gamma=\Gamma_2$: proof of Theorem \ref{Bocher1}}
\

\begin{proof}
	We write $v(x)=v(|x|)=\min_{\partial B_{|x|}}u$. Then $v$ belongs to $LSC(B_1\backslash\{0\})\cap L_{loc}^{\infty}(B_1\backslash\{0\})$ and satisfies $\lambda(A^v)\in\bar\Gamma$ in $B_1\backslash\{0\}$ in the viscosity sense. 
	
	We claim that
	\begin{align*}
	\text{either }v(x)=-4\ln |x|+C\text{ for some constant }C\text{ or }\sup_{B_{1/2}\backslash\{0\}}v<\infty.
	\end{align*}
	Indeed, if the first alternative does not hold, we can find $0<r_1<r_2<1$ such that
	\begin{align*}
	v(r_1)+4\ln r_1\neq v(r_2)+4\ln r_2.
	\end{align*}
	By Lemma \ref{radbehavior}, 
	\begin{align*}
	v(r_2)\leq v(r_1)<4(\ln r_2-\ln r_1)+v(r_2).
	\end{align*}
	Therefore, we can find function
	\begin{align*}
	\hat{v}(r)=-\frac{4}{p-1}\ln(r^{p-1}+C_1)+C_2
	\end{align*}
	satisfying $\hat{v}(r_1)=v(r_1)$ and $\hat{v}(r_2)=v(r_2)$, $C_1>0$. By Theorem \ref{radialsol}, $\lambda(A^{\hat{v}})\in\partial\Gamma$ in $B_1\backslash\{0\}$. By Corollary \ref{comp3}, we have $v\leq\hat{v}$ in $(0,r_1)$. In particular, $v$ is bounded near the origin. The claim is proved.
	
	If the first alternative holds, we have $u\geq v$ in $B_1\backslash\{0\}$, $\Delta u\leq0=\Delta v$ in $B_1\backslash\{0\}$, and the set $\{x\in B_1\backslash\{0\}:u=v\}$ is non-empty. By strong maximum principle, $u\equiv v$ and the conclusion follows.
	
	If the second alternative holds, then conclusion follows from Proposition \ref{removable2}.
\end{proof}

\subsubsection{The case $\Gamma=\Gamma_p$, $1<p<2$: proof of Theorem \ref{Bocher2}}
\ 

Now we consider the case $\Gamma\neq\Gamma_2$. By Proposition \ref{removable2}, it suffices to assume
\begin{align*}
\liminf_{|x|\rightarrow 0}(u(x)+4\ln|x|)>-\infty.
\end{align*}

\begin{lemm}\label{bocher2}
	Let $\Gamma=\Gamma_p$ for some $1<p<2$. Let $u\in C_{loc}^{0,1}(B_1\backslash\{0\})$ be a viscosity solution of
	\begin{align*}
	\lambda(A^u)\in\partial\Gamma
	\end{align*}
	in $B_1\backslash\{0\}$, and $u$ satisfy $\liminf_{|x|\rightarrow 0}(u(x)+4\ln|x|)>-\infty$.
	Then
	\begin{align*}
	\lim_{|x|\rightarrow 0}(u(x)+4\ln|x|)
	\end{align*}
	exists and is a finite number.
\end{lemm}

\begin{proof}
	We still define that
	\begin{align*}
	v(r)=\min_{\partial B_r}u.
	\end{align*}
	Then $v$ is superharmonic in $B_1\backslash\{0\}$. Since $\{0\}$ has zero Newtonian capacity, $v$ is superharmonic in $B_1$. In particular, $v$ is non-increasing.
	
	Fix some $0<\rho_1<1$ and for $0<\rho<\rho_1$, $j\in \mathbb{N}_{+}$, let $w_{\rho,j}$ be the radially symmetric function which is of the form
	\begin{align*}
	w_{\rho,j}(r)=\frac{4(2-p)}{p-1}\ln(r^{1-\frac{1}{2-p}}-a_{\rho,j})+b_{\rho,j}
	\end{align*}
	in $B_1\backslash\{0\}$ such that $w_{\rho,j}(\rho)=v(\rho)+\frac{1}{j}$ and $w_{\rho,j}(\rho_1)=v(\rho_1)$. After solving the equation, we obtain that
	\begin{align*}
	a_{\rho,j}&=\frac{-\rho^{-(p-1)/(2-p)}+{\rho_1}^{-(p-1)/(2-p)}e^{\frac{p-1}{4(2-p)}(w_\rho(\rho)-w_\rho(\rho_1))}}{e^{\frac{p-1}{4(2-p)}(w_\rho(\rho)-w_\rho(\rho_1))}-1},\\
	b_{\rho,j}&=w_{\rho,j}(\rho)-\frac{4(2-p)}{p-1}\ln\frac{\rho^{-(p-1)/(2-p)}-\rho_1^{-(p-1)/(2-p)}}{1-e^{-\frac{p-1}{4(2-p)}(w_\rho(\rho)-w_\rho(\rho_1))}}.
	\end{align*}
	Claim: $w_{\rho,j}(r)\geq v(r)$ for all $0<r<\rho$.
	
	In fact, by Theorem \ref{radialsol}, we know if $a_{\rho,j}\geq 0$ then $\lambda(A^{w_{\rho,j}})\in\partial\Gamma$, if $a_{\rho,j}<0$ then by direct computation, $\lambda(A^{w_{\rho,j}})\in\mathbb{R}^2\backslash\bar\Gamma$. Therefore, $w_{\rho,j}$ is a subsolution. If $w_{\rho,j}(s)<v(s)$ for some $s<\rho$, then the comparison principle implies that $w_{\rho,j}(r)\leq v(r)$ for $s<r<\rho_1$, which implies in particular that $w_{\rho,j}(\rho)\leq v(\rho)$, contradicting our choice of $w_{\rho,j}(\rho)$. The claim is proved.
	
	It follows that
	\begin{align*}
	\limsup_{r\rightarrow 0}(v(r)+4\ln r)\leq\limsup_{r\rightarrow 0}(w_{\rho,j}(r)+4\ln r)=b_{\rho,j}\text{ for all }0<\rho<\rho_1\text{ and }j\in \mathbb{N}_{+}.
	\end{align*}
	In particular,
	\begin{align*}
	\limsup_{r\rightarrow 0}(v(r)+4\ln r)<+\infty.
	\end{align*}
	Since $\liminf_{|x|\rightarrow 0}(u(x)+4\ln|x|)>-\infty$, we obtain
	\begin{align*}
	\liminf_{\rho\rightarrow 0}b_{\rho,j}&=\liminf_{\rho\rightarrow 0}(v(\rho)+\frac{1}{j}-\frac{4(2-p)}{p-1}\ln\frac{\rho^{-(p-1)/(2-p)}-\rho_1^{-(p-1)/(2-p)}}{1-e^{-\frac{p-1}{4(2-p)}(w_\rho(\rho)-w_\rho(\rho_1))}})\\
	&=\liminf_{\rho\rightarrow 0}(v(\rho)+4\ln\rho)+\frac{1}{j}>-\infty.
	\end{align*}
	Therefore,
	\begin{align*}
	\limsup_{r\rightarrow 0}(v(r)+4\ln r)\leq\liminf_{\rho\rightarrow 0}(v(\rho)+4\ln\rho)+\frac{1}{j}.
	\end{align*}
	Send $j\rightarrow\infty$, it follows that $\lim_{r\rightarrow 0}(v(r)+4\ln r)$ exists and is finite.
	
	We thus have
	\begin{align*}
	a:=\liminf_{|x|\rightarrow 0}(u(x)+4\ln|x|)=\lim_{r\rightarrow 0}(v(r)+4\ln r).
	\end{align*}
	We next claim that
	\begin{align*}
	A:=\limsup_{|x|\rightarrow 0}(u(x)+4\ln|x|)\text{ is finite.}
	\end{align*}
	To prove the claim, for $0<r<1/4$, let
	\begin{align*}
	u_r(y)=u(ry),\quad \frac{1}{2}<|y|<2.
	\end{align*}
	Then $u_r$ satisfies $\lambda(A^{u_r})\in\partial\Gamma$ in $\{1/2<|y|<2\}$. Thus, by Theorem \ref{grad3},
	\begin{align*}
	\max_{\partial B_1}u_r\leq C+\min_{\partial B_1}u_r,
	\end{align*}
	where $C$ depends only on $n$. Equivalently,
	\begin{align}\label{Harnack}
	\max_{\partial B_r}u\leq C+\min_{\partial B_r}u.
	\end{align}
	The claim follows from above.
	
	By Proposition \ref{Asymptotic behavior-1}, we know $a\geq 0$. Moreover, (\ref{Harnack}) implies that $A\leq a+C<\infty$.
	
	Next we show $A=a$. Assume by contradiction that $A>a$. Then for some $\epsilon>0$, we can find a sequence $x_j\rightarrow 0$ such that
	\begin{align*}
	u(x_j)+4\ln|x_j|\geq a+2\epsilon.
	\end{align*}
	Furthermore, we can assume that
	\begin{align*}
	\min_{\partial B_{|x_j|}}u+4\ln|x_j|=v(x_j)+4\ln|x_j|\leq a+\epsilon.
	\end{align*}
	Define
	\begin{align*}
	u_j(y)=u(\frac{y}{R_j})+4\ln\frac{1}{R_j},\quad |y|<R_j=|x_j|^{-1}\rightarrow\infty.
	\end{align*}
	Then we have
	\begin{align*}
	\lambda(A^{u_j})\in\partial \Gamma &\text{ in }B_{R_j}\backslash\{0\},\\
	\min_{\partial B_1}u_j\leq a+\epsilon &\text{ and }\max_{\partial B_1}u_j\geq a+2\epsilon.
	\end{align*}
	Since $\min_{\partial B_1}u_j$ is bounded, we can apply Theorem \ref{grad3} to obtain the boundedness of $u_j$ and $|\nabla u_j|$ on every compact subset of $\mathbb{R}^2\backslash\{0\}$. By the Ascoli-Arzela theorem, after passing to a subsequence, $u_j$ converges uniformly on compact subset of $\mathbb{R}^2\backslash\{0\}$ to some locally Lipschitz function $u_{*}$. Furthermore, $u_{*}$ satisfies the equation $\lambda(A^{u_{*}})\in\partial\Gamma$ in the viscosity sense.
	
	By Theorem \ref{dege1}, $u_{*}$ is a radially symmetric function, i.e. $u_{*}(y)=u_{*}(|y|)$. This results in a contradiction to $\min_{\partial B_1}u_j\leq a+\epsilon$, $\max_{\partial B_1}u_j\geq a+2\epsilon$ and convergence of $u_j$ to $u_{*}$. We conclude that $A=a$.
\end{proof}

\begin{lemm}\label{bocher3}
	Let $\Gamma=\Gamma_p$ for some $1<p<2$. Let $u\in C_{loc}^{0,1}(B_1\backslash\{0\})$ be a viscosity solution of
	\begin{align*}
	\lambda(A^u)\in\partial\Gamma \text{ in }B_1\backslash\{0\},
	\end{align*}and
	$u$ satisfies $\liminf_{|x|\rightarrow 0}(u(x)+4\ln|x|)>-\infty$.
	Then
	\begin{align*}
	u=\dfrac{4(2-p)}{p-1}\ln(r^{-(p-1)/(2-p)}+\mathring w)+a,
	\end{align*}
	where $\mathring w$ is a nonpositive function in $L_{loc}^{\infty}(B_1)$, and $a=\sup_{B_1\backslash\{0\}}(u(x)+4\ln |x|)<+\infty$. Moreover,
	\begin{align}\label{wmaxprin}
	\min_{\partial B_r}\mathring  w\leq \mathring w\leq \max_{\partial B_r} \mathring w \text{ in }B_r\backslash\{0\},\quad \forall\ 0<r<1.
	\end{align}
\end{lemm}

\begin{proof}
	By Lemma \ref{bocher2}, $a=\sup_{B_1\backslash\{0\}}(u(x)+4\ln |x|)<+\infty$. Since $u(x)+4\ln |x|\leq a$ in $B_1\backslash\{0\}$, it follows that $\mathring w\leq 0$ in $B_1\backslash\{0\}$. Now we will show that
	\begin{align*}
	\min_{\partial B_r} \mathring w\leq \mathring w\leq \max_{\partial B_r} \mathring w \text{ in }B_r\backslash\{0\},\quad \forall\ 0<r<1.
	\end{align*}
	Fix $0<r<1$, for $\epsilon>0$, set
	\begin{align*}
	v_{\epsilon,r}^+(x)=\dfrac{4(2-p)}{p-1}\ln(r^{-(p-1)/(2-p)}+\max_{\partial B_r}\mathring w)+a+\epsilon,\\
	v_{\epsilon,r}^-(x)=\dfrac{4(2-p)}{p-1}\ln(r^{-(p-1)/(2-p)}+\min_{\partial B_r}\mathring w)+a-\epsilon.
	\end{align*}
	By Theorem \ref{radialsol}, we have $\lambda(A^{v_{\epsilon,r}^{+}})\in\partial\Gamma$ and $\lambda(A^{v_{\epsilon,r}^{-}})\in\partial\Gamma$ in $B_r\backslash\{0\}$, and $v_{\epsilon,r}^{-}<u<v_{\epsilon,r}^{+}$ on $\partial B_r$. Furthermore, by Theorem \ref{bocher2}, there exists $\delta=\delta(\epsilon,r)>0$ such that
	\begin{align*}
	v_{\epsilon,r}^{-}<u<v_{\epsilon,r}^{+}\text{ in } B_\delta\backslash\{0\}.
	\end{align*}
	Thus, by Proposition \ref{Comparison-2},
	\begin{align*}
	v_{\epsilon,r}^{-}\leq u\leq v_{\epsilon,r}^{+}\text{ in } B_r\backslash\{0\}.
	\end{align*}
	Sending $\epsilon\rightarrow 0$, we obtain (\ref{wmaxprin}).

\end{proof}
Theorem \ref{Bocher2} follows from Lemma \ref{bocher2} and Lemma \ref{bocher3}.

\medskip

\section{Existence Result}\label{existence}
In this section, we prove the existence result Theorem \ref{Niren}.
\medskip
\subsection{Kazdan-Warner type identity}
\ 

We first establish Theorem \ref{KW}, the Kazdan-Warner type identity  on $(\mathbb{S}^2,g)$ for the $\sigma_2$-Nirenberg problem.

Recall that
\begin{align*}
A_{g_u}:=-\nabla_{g}^2u+\frac{1}{2}du\otimes du-\frac{1}{4}|\nabla_{g} u|^2g+g,
\end{align*}

For any $2\times 2$ symmetric matrix $\Lambda$, denote the $k$-th Newton transformation $T_k(\Lambda)=\sum_{j=0}^k (-1)^j\sigma_{k-j}(\Lambda)\Lambda^j$, namely, $T_0=\delta_{ij}$, $T_1=\sigma_1(\Lambda)\delta_{ij}-\Lambda$.  Then we have
\begin{align}\label{Tk}
(k+1)\sigma_{k+1}(\Lambda)=T_k(\Lambda)_b^a \Lambda_a^b,\text{ for }k=0,1.
\end{align}

\begin{prop}\label{identity}
	Let $u$ be a smooth function on $\mathbb{S}^2$. On $(\mathbb{S}^2,g_u)$ where $g_u=e^ug$, we have\\
	(i) $\nabla_c A_{ab}=\nabla_b A_{ac}$.\\
	(ii) $\nabla_a T_k({g_u}^{-1}A_{g_u})_b^a=0$, $k=0,1$.
	\\
	Here the covariant derivative is taken with respect to $g_u$.
\end{prop}
\begin{proof}
	\
	
	We prove (i) here. (ii) follows from (i) as in \cite{R}.
	
	It suffices to show $\nabla_1 A_{12}=\nabla_2 A_{11}$ and $\nabla_2 A_{12}=\nabla_1 A_{22}$. We will only prove $\nabla_1 A_{12}=\nabla_2 A_{11}$, and the other one is similar.
	
	Set up local coordinates $\{x_1,x_2\}$ such that $g=e^{\phi}\delta_{ij}$, where $\phi(x)=\ln\frac{4}{(1+|x|^2)^2}$. Therefore, $g_u=e^{u+\phi}\delta_{ij}$.
	
	By definition (\ref{Agu}),
	\begin{align}\label{Aij}
	A_{ij}=-u_{ij}+{\Gamma}_{ij}^ku_k+\frac{1}{2}u_iu_j-\frac{1}{4}(u_1^2+u_2^2)\delta_{ij}+e^\phi\delta_{ij}
	\end{align}
	where $\Gamma_{ij}^k$ is the Christoffel symbols with respect to $g$.
	\begin{align*}
	\nabla_1 A_{12}&=\partial_1(A_{12})-A(\nabla_1 e_1,e_2)-A(e_1,\nabla_1 e_2)\\
	&=\partial_1(A_{12})-\tilde\Gamma_{11}^1 A_{12}-\tilde\Gamma_{11}^2 A_{22}-\tilde\Gamma_{12}^1 A_{11}-\tilde\Gamma_{12}^2 A_{12},
	\\
	\nabla_2 A_{11}&=\partial_2(A_{11})-2A(\nabla_2 e_1,e_1)\\
	&=\partial_2(A_{11})-2\tilde\Gamma_{12}^1 A_{11}-2\tilde\Gamma_{12}^2 A_{12},
	\end{align*}
	where $\tilde\Gamma_{ij}^k$ is the Christoffel symbols with respect to $g_u$.
	
	By direct computations, $\Gamma_{11}^1=\Gamma_{12}^2=\phi_1/2$, $\Gamma_{11}^2=-\phi_2/2$,  $\Gamma_{12}^1=\Gamma_{22}^2=\phi_2/2$, $\Gamma_{22}^1=-\phi_1/2$, and $\tilde\Gamma_{ij}^k$ is given the formula of $\Gamma_{ij}^k$ with $\phi$ replaced by $\phi+u$.

	Therefore,
	\begin{align*}
	&\nabla_1 A_{12}-\nabla_2 A_{11}
	\\=&(\partial_1(A_{12})-\tilde\Gamma_{11}^1 A_{12}-\tilde\Gamma_{11}^2 A_{22}-\tilde\Gamma_{12}^1 A_{11}-\tilde\Gamma_{12}^2 A_{12})-(\partial_2(A_{11})-2\tilde\Gamma_{12}^1 A_{11}-2\tilde\Gamma_{12}^2 A_{12})
	\\=&\partial_1(A_{12})-\partial_2(A_{11})-\tilde\Gamma_{11}^2 A_{22}+\tilde\Gamma_{12}^1 A_{11}
	\\=&\partial_1(A_{12})-\partial_2(A_{11})+(u_2+\phi_2)(A_{11}+A_{22})/2
	\end{align*}
	
	Using (\ref{Aij}) and the formula of $\Gamma_{ij}^k$, we obtain
	\begin{align*}
	A_{12}&=-u_{12}+\phi_2u_1/2+\phi_1u_2/2+u_1u_2/2,\\
	A_{11}&=-u_{11}+\phi_1u_1/2-\phi_2u_2/2+u_1^2/4-u_2^2/4+e^\phi,\\
	A_{11}&+A_{22}=-u_{11}-u_{22}+2e^\phi.
	\end{align*}
	
	It follows that
	\begin{align*}
	&\nabla_1 A_{12}-\nabla_2 A_{11}
	\\=&(-u_{112}+\phi_{12}u_1/2+\phi_{2}u_{11}/2+\phi_{11}u_2/2+\phi_1u_{12}/2+u_{11}u_2/2+u_1u_{12}/2)\\&-(-u_{112}+\phi_{12}u_1/2+\phi_1u_{12}/2-\phi_{22}u_2/2-\phi_2u_{22}/2+u_1u_{12}/2-u_2u_{22}/2+e^\phi\phi_2)\\&+(u_2+\phi_2)(-u_{11}-u_{22}+2e^\phi)/2
	\\=&(\frac{1}{2}(\phi_{11}+\phi_{22})+e^\phi)u_2=0.
	\end{align*}
\end{proof}

\begin{proof}[Proof of Theorem \ref{KW}]
	
	We claim that
	\begin{align}\label{k=2}
	\langle X,\nabla\sigma_2\rangle=\nabla_a(T_b^a\nabla^b(\text{div} X)+2\sigma_2X^a),
	\end{align}
	where $T_b^a$ denotes the components of $T_1$.
	
	Once the claim is true, Theorem \ref{KW} for $k=2$ follows by integrating (\ref{k=2}) over $\mathbb{S}^2$.
	\\
	Proof of the claim:
	
	Let $\phi_t$ denote the local one-parameter family of conformal diffeomorphisms of $(\mathbb{S}^2,g)$ generated by $X$. Thus for some function $u_t$, we have $\phi_t^*(g)=e^{u_t}g=:g_t$. We have the following properties:
	\begin{align}\label{k21}
	\sigma_2({g}^{-1}A_{g})\circ\phi_t=\sigma_2(g_t^{-1}A_{g_t}).
	\end{align}
	\begin{align}\label{k22}
	\dot{u}:=\frac{d}{dt}\bigg|_{t=0}u_t=\text{div}X=\nabla_aX^a.
	\end{align}
	\begin{align}\label{k23}
	\frac{d}{dt}\bigg|_{t=0}(g_t^{-1}A_{g_t})_b^a=-\nabla_b^a\dot{u}-\dot{u}A_b^a.
	\end{align}
	Using (\ref{Tk}),(\ref{k21}),(\ref{k22}),(\ref{k23}), and Proposition \ref{identity},
	\begin{align*}
	\langle X,\nabla\sigma_2\rangle&=T_a^b(-\nabla_b^a\dot{u}-\dot{u}A_b^a)\\
	&=-T_a^b\nabla_b^a\dot{u}-2\sigma_2\dot{u}\\
	&=-T_a^b\nabla_b^a\dot{u}-2\sigma_2\nabla_bX^b\\
	&=-T_a^b\nabla_b^a\dot{u}+2\langle X,\nabla\sigma_2\rangle-2\nabla_b(\sigma_2X^b)\\
	&=-\nabla_b(T_a^b\nabla^a\dot{u}+2\sigma_2X^b)+2\langle X,\nabla\sigma_2\rangle.
	\end{align*}
	Claim (\ref{k=2}) follows. 
\end{proof}

\medskip

\subsection{One point blow-up phenomena}

\ 

\medskip

\begin{theo}\label{onept}
	Let $(f,\Gamma_p)$ satisfy (\ref{f1})-(\ref{f5}), and $1<p\leq 2$, $\{K_i\}$ be a sequence of positive $C^2$ functions on $\mathbb{S}^2$ satisfying for some positive constants $c_1$ and $C_2$, $\min_{\mathbb{S}^2} K_i\geq c_1>0$ and $\sup \|K_i\|_{C^2(\mathbb{S}^2)}\leq C_2<\infty$ for all $i$, and let $\{u_i\}\in C^2$ be a sequence of functions satisfying
	\begin{align*}
	f(\lambda(g_{u_i}^{-1}A_{g_{u_i}}))=K_i,\ \lambda(g_{u_i}^{-1}A_{g_{u_i}})\in \Gamma_p\quad\text{on }\mathbb{S}^2,
	\end{align*}
	and, with $x_i\in\mathbb{S}^2$,
	\begin{align*}
	&u_i(x_i)=\max_{\mathbb{S}^2} u_i\rightarrow\infty.
	\end{align*}
	Then for some constant $C>0$ depending only on $p$, $c_1$ and $C_2$,
		\begin{align*}
		u_i(x)\leq -2 \ln d_g(x,x_i)+C\text{ for all }x\in \mathbb{S}^2\backslash\{x_i\}.
		\end{align*}
	Here $d_g(x,x_i)$ denotes the distance between $x$ and $x_i$ in the metric $g$.
\end{theo}
We start with a lemma, see e.g. \cite[page 271]{HaE}, \cite[page 135]{Chavel}.
\begin{lemm}\label{singu}
	Let $(N, g)$ be a two dimensional complete smooth Riemannian manifold with smooth
	boundary $\partial N$. If,  for some constants $c_0>\alpha\geq 0$, $K_g\geq-\alpha^2$ on $N$ and if the geodesic curvature $\kappa$ of $\partial N$ with respect to its inner normal satisfies $\kappa>c_0$ on $\partial N$,
	then	\begin{align*}
	d_g(x,\partial N)\leq U(\alpha, c_0)\text{ for all }x\in N,
	\end{align*}
	where $d_g$ denotes the distance function induced by $g$ and
	\begin{align*}
	U(\alpha, c_0)=\begin{cases}\frac{1}{c_0},\quad \ \ & 
	\text{if }\alpha=0,\\ 
	\frac{1}{\alpha}\coth^{-1}(\frac{c_0}{\alpha}), \quad \ \ & \text{if }\alpha>0.
	\end{cases}
	\end{align*}
\end{lemm}

\begin{proof}[Proof of Theorem \ref{onept}]
	The proof is analogous to the proof of \cite[Lemma 3.1]{LN2}.

	\begin{lemm}\label{vol}
		Assume for some $C_1\geq 0$, $M_i\rightarrow \infty$ and $y_i\in \mathbb{S}^2$ that
		\begin{align*}
		u_i(y_i)\rightarrow\infty\text{ and }\sup\{u_i(y):d_g(y,y_i)\leq M_i e^{-\frac{u_i(y_i)}{2}} \}\leq u_i(y_i)+C_1.
		\end{align*}
		Then for any $0<\mu<1$, there exists $M=M(C_1,\mu)$, such that for all sufficiently large $i$,
		\begin{align*}
		\text{Vol}_{g_{u_i}}(\{y:d_g(y,y_i)\leq M e^{-\frac{u_i(y_i)}{2}} \})\geq (1-\mu)\text{Vol}_{g_{u_i}}(\mathbb{S}^2).
		\end{align*}
	\end{lemm}

	\begin{proof}
		Without loss of generality, assume $f(\lambda(g^{-1}A_{g}))=1$ on $S^2$, where $g$ is the standard metric on $S^2$.
			
		For $q\in\mathbb{R}^2$, $a>0$, $c=\ln 4$, define
		\begin{align*}
		U_{a,q}(x)=2\ln \frac{a}{a^2+|x-q|^2}+c,\quad x\in\mathbb{R}^2.
		\end{align*}
		Write $\mathbb{S}^2=\{(z_1,z_2,z_3)\in\mathbb{R}^3|z_1^2+z_2^2+z_3^2=1\}$.
		
		Let $(x_1,x_2)\in\mathbb{R}^2$ be the stereographic projection coordinates of $z\in S^2$, i.e.
		\begin{align*}
		z_i=\frac{2x_i}{1+|x|^2}\text{ for }1\leq i\leq 2,\text{ and }z_3=\frac{|x|^2-1}{|x|^2+1}.
		\end{align*}
		Then
		\begin{align*}
		g=|dz|^2=(\frac{2}{1+|x|^2})^2|dx|^2=e^{U_{1,0}}|dx|^2.
		\end{align*}
		A calculation gives
		\begin{align*}
		(e^{U_{a,q}}g_{\text{flat}})^{-1}A_{e^{U_{a,q}}g_{\text{flat}}}\equiv I,
		\end{align*}
		and
		\begin{align*}
		f(\lambda(A_{e^{U_{a,q}}g_{\text{flat}}}))=1\text{ on }\mathbb{R}^2,
		\end{align*}
		where $g_{\text{flat}}=|dx|^2$ is the standard metric on $\mathbb{R}^2$.
		
		Define a map $\Phi_i:\mathbb{R}^2\approx T_{y_i}(\mathbb{S}^2,g)\rightarrow \mathbb{S}^2$ by:
		\begin{align*}
		\Phi_i(x)=\exp_{y_i}e^{\frac{c-u_i(y_i)}{2}}x,
		\end{align*}
		and let
		\begin{align*}
		\tilde{u}_i(x)=u_i\circ\Phi_i(x)+c-u_i(y_i),\ x\in\mathbb{R}^2.
		\end{align*}
		Then $\tilde{u}_i$ satisfies 
		\begin{align*}
		f(\lambda(A_{e^{{\tilde{u}_i}}\tilde{h}_i}))=1\text{ and }\lambda(A_{e^{{\tilde{u}_i}}\tilde{h}_i})\in\Gamma \text{ on }\{|x|< \pi e^{\frac{u_i(y_i)-c}{2}}\}.
		\end{align*}
		where $\tilde{h}_i=e^{u_i(y_i)-c}\Phi_i^*g$. Now $\tilde{h}_i\rightarrow g_{\text{flat}}$ on $C_{loc}^3(\mathbb{R}^2)$. Furthermore, $\tilde{u}_i(0)=c$, and by assumption, $\tilde{u}_i\leq C_1+c$. By local derivatives estimates Theorem \ref{grad2} and Theorem \ref{C2},  $\tilde{u}_i$ is uniformly bounded in $C^2$ on any compact set of $\mathbb{R}^2$. By Nirenberg's estimate \cite{N2}, $\tilde{u}_i$ subconverges in $C_{loc}^{2,\alpha}(\mathbb{R}^2)$ to some function $\tilde{u}_*\in C^2(\mathbb{R}^2)$ which satisfies
		\begin{align*}
		f(\lambda(A_{e^{\tilde{u}_*}g_{\text{flat}}}))=1\text{ and }\lambda(A_{e^{\tilde{u}_*}g_{\text{flat}}})\in\Gamma \text{ on }\mathbb{R}^2.
		\end{align*}
		By the Liouville theorem in our earlier paper \cite{LLL2}, we have $\tilde{u}_*=U_{a_*,x_*}$  for some $a_*>0$ and $x_*\in\mathbb{R}^2$. Since $\tilde{u}_*(0)=\lim \tilde{u}_i(0)=c$ and $\tilde{u}_*\leq C_1+c$, we have, for some constant $C$ depending only on $C_1$,
		\begin{align*}
		|x_*|\leq C\text{ and }C^{-1}\leq a_*\leq C.
		\end{align*}
		In particular, for any $R>0$ and $\mu>0$,
		\begin{align*}
		\| \tilde{u}_i-\tilde{u}_*\|_{C^2(\bar{B_R})}\leq \mu\text{ for all sufficiently large }i.
		\end{align*}
		It follows that the metrics $e^{\tilde{u}_i}\tilde{h}_i$ converge on compact subsets to the metric $e^{\tilde{u}_*}g_{\text{flat}}$. Since $(B(0,r),e^{\tilde{u}_i}\tilde{h}_i)$ is isometric to $(\Phi_i(B(0,r)),g_{u_i})$, for any $r>0$, we obtain:
		
		For any $\epsilon>0$, there exists $R=R(\epsilon,C_1)>0$ such that
		
		(i) $|\text{Vol}_{g_{u_i}}(\Phi_i(B(0,R)))-\text{Vol}(S^2,g)|\leq C\epsilon^2$ for some $C$ independent of $i$ and $\epsilon$,
		
		(ii) the curvature of the hypersurface $\partial\Phi_i(B(0,R))$ with respect to $g_{u_i}$ and the unit normal pointing away from $\Phi_i(B(0,R))$ is no smaller than $\frac{1}{\epsilon}$.
		
		Using (ii) and Lemma \ref{singu}, we see that
		\begin{align*}
		\text{diam}_{g_{u_i}}(\mathbb{S}^2\backslash\Phi_i(B(0,R)))\leq C\epsilon.
		\end{align*}
		By Bishop's theorem and $K_{g_{u_i}}\geq K_0>0$, this implies
		\begin{align*}
		\text{Vol}_{g_{u_i}}(\mathbb{S}^2\backslash\Phi_i(B(0,R)))\leq C\epsilon^2.
		\end{align*}
		Lemma \ref{vol} is proved.
	\end{proof}
	Now to prove Theorem \ref{onept}, we will show that, for some constant $C>0$ independent of $i$,
	\begin{align*}
	u_i(x)\leq -2 \ln d_g(x,x_i)+C\text{ for all }x\in \mathbb{S}^2\backslash\{x_i\}.
	\end{align*}
	Suppose not, then for some $\tilde{x}_i\in \mathbb{S}^2\backslash\{x_i\}$,
	\begin{align*}
	u_i(\tilde{x}_i)+2\ln d_g(x_i,\tilde{x}_i)=\max_{\mathbb{S}^2} (u_i+2\ln d_g(x_i,\cdot))\rightarrow\infty.
	\end{align*}
	Since $(\mathbb{S}^2,g)$ is compact, this implies $u_i(\tilde{x}_i)\rightarrow\infty$.
	
	Apply Lemma \ref{vol} to $C_1=0$, $y_i=x_i$, and $M_i=\delta e^{\frac{u_i(x_i)}{2}}$ with some small $\delta=\delta(\mathbb{S}^2,g)$, we find
	\begin{align*}
	\text{Vol}_{g_{u_i}}(\{y:d_g(y,x_i)\leq M e^{-\frac{u_i(x_i)}{2}}\})\geq \frac{3}{4}\text{Vol}_{g_{u_i}}(\mathbb{S}^2).
	\end{align*}
	where $M$ is some universal constant.
	
	Apply Lemma \ref{vol} to $C_1=2\ln 2$, $y_i=\tilde{x}_i$, and $M_i=\frac{1}{2}d(x_i,\tilde{x}_i) e^{\frac{u_i(\tilde{x}_i)}{2}}$, we find
	\begin{align*}
	\text{Vol}_{g_{u_i}}(\{y:d_g(y,\tilde{x}_i)\leq \tilde{M} e^{-\frac{u_i(\tilde{x}_i)}{2}}\})\geq \frac{3}{4}\text{Vol}_{g_{u_i}}(\mathbb{S}^2).
	\end{align*}
	where $\tilde{M}$ is some universal constant. On the other hand, since $u_i(x_i)\geq u_i(\tilde{x}_i)$, and $u_i(\tilde{x}_i)+2\ln d_g(x_i,\tilde{x}_i)\rightarrow\infty$, we know the sets
	\begin{align*}
	\{y:d_g(y,x_i)\leq M e^{-\frac{u_i(x_i)}{2}}\}\text{ and }\{y:d_g(y,\tilde{x}_i)\leq \tilde{M} e^{-\frac{u_i(\tilde{x}_i)}{2}}\}
	\end{align*}
	are disjoint for sufficiently large $i$. This is a contradiction.
\end{proof}

\medskip

\subsection{$C^0$ estimate}
\

In this section, we establish the following $C^0$ estimate.

\begin{prop}\label{C0esti}
	Assume that $K$ satisfies the nondegeneracy condition (\ref{Knondege}) and $v\in C^2$ satisfies equation (\ref{eqn}).  Then
	\begin{align*}
	v\leq C,
	\end{align*}
	where $C$ depends only on an upper bound of $|\ln K|$, a positive lower bound of $\frac{1}{\|\nabla^2K_i\|_{C^0(\mathbb{S}^2)}}(|\nabla K_i|+|\Delta K_i|)$, and the modulus of continuity of $\frac{1}{\|\nabla^2K_i\|_{C^0(\mathbb{S}^2)}}\nabla^2 K_i$.
\end{prop}
\begin{rema}
	The $C^0$ estimates of the $\sigma_k$-Nirenberg problem under such nondegeneracy condition on $K$ were obtained in \cite{CHY} for $\sigma_2$ on $\mathbb S^4$, in \cite{LNW1} for $\sigma_k$, $n/2\leq k\leq n$, on $\mathbb{S}^n$, $n\geq 3$, and in \cite{LNW2} for $\sigma_k$, $2\leq k<n/2$, on $\mathbb S^n$ for axisymmetric functions $K$. For such detailed dependence of the constant $C$ on $K$ in dimensions $n\geq 3$, see \cite{LNW1} and \cite{LNW2}.
\end{rema}
\begin{proof}	
	Suppose not, then
	
	(i) there exist positive $C^2$ functions $K_i$ such that $|\ln K_i|$ is uniformly bounded, $\frac{1}{\|\nabla^2K_i\|_{C^0(\mathbb{S}^2)}}(|\nabla K_i|+|\Delta K_i|)$ is uniformly bounded from below by a positive constant, and $\frac{1}{\|\nabla^2K_i\|_{C^0(\mathbb{S}^2)}}\nabla^2 K_i$ is equicontinuous,
	
	(ii) there exist $C^2$ functions $v_i$ satisfying equation (\ref{eqn}) with $K$ replaced by $K_i$ and a sequence of points $\{P_i\}$ such that $v_i(P_i)=\max_{\mathbb{S}^2} v_i\rightarrow\infty$ as $i\rightarrow\infty$.
	
	We assume without loss of generality that $P_i=P$ for all $i$, and $P$ is the south pole.
	
	We claim that $\|\nabla^2 K_i\|_{C^0(\mathbb{S}^2)}$ is uniformly bounded. Indeed, if $\|\nabla^2 K_{i_j}\|_{C^0(\mathbb{S}^2)}$ is a subsequence going to $\infty$, then $\frac{1}{\|\nabla^2 K_{i_j}\|_{C^0(\mathbb{S}^2)}}K_{i_j}$ converges uniformly to $0$. By Ascoli-Arzela's theorem, this sequence is precompact in $C^2$, hence the convergence would hold in $C^2$. This contradicts with the assumption that $\frac{1}{\|\nabla^2K_i\|_{C^0(\mathbb{S}^2)}}(|\nabla K_i|+|\Delta K_i|)$ is uniformly bounded from below by a positive constant.
	
	Therefore, by (i) and the above claim, we may assume without loss of generality that $K_i$ converges in $C^2$ to some positive $C^2$ function $K_\infty$.
	
	Recall Theorem \ref{onept}, there exists $C>0$ such that
	\begin{align*}
	v_i(x)\leq C-2 \ln d_g(x,P)\text{ for all }x\in \mathbb{S}^2\backslash\{P\}.
	\end{align*}
	Let $\Phi:\mathbb{R}^2\rightarrow\mathbb{S}^2$ be the inverse of stereographic projection. $P$ is the south pole.
	\begin{align*}
	x_i&=\frac{2y_i}{1+|y|^2}, i=1,2,\\
	x_3&=\frac{|y|^2-1}{|y|^2+1}.
	\end{align*}
	Then we have
	\begin{align*}
	\sigma_2(\lambda(A^{u_i}))=K_i(\Phi(y)),\  \lambda(A^{u_i})\in\Gamma_2,\quad\text{in }\mathbb{R}^2,
	\end{align*}
	where
	\begin{align*}
	u_i(y)=v_i(x)+2\ln\frac{1}{|y|^2+1}.
	\end{align*}
	We know that $0$ is a maximum point of $u_i$, $u_i(0)\rightarrow\infty$ as $i\rightarrow\infty$.

	Let $\lambda_i=\exp(u_i(0)/2)$, and define
	\begin{align*}
	\tilde{u}_i(z)=u_i(\lambda_i^{-1}z)-u_i(0).
	\end{align*}
	Then
	\begin{align*}
	\sigma_2(\lambda(A^{\tilde{u}_i}))=K_i(\Phi(\lambda_i^{-1}z)), \ \lambda(A^{\tilde{u}_i})\in\Gamma_2,\quad\text{in }\mathbb{R}^2,
	\end{align*}
	Passing to a subsequence and apply Liouville type theorem, for every $\epsilon_i\rightarrow 0^{+}$, we can find $R_i\rightarrow\infty$ with $R_i<\epsilon_i^{-1}$ and $R_i/\lambda_i\rightarrow 0$ as $i\rightarrow\infty$, such that
	\begin{align}\label{estimateui}
	\| \tilde{u}_i(z)-2\ln\frac{1}{|z|^2+1}\|_{C^2(B_{2R_i})}<\epsilon_i.
	\end{align}
	
	Now we prove a lemma which gives the optimal decay estimate of $u_i$.
	\begin{lemm}\label{optimaldecay}
		
		There exists a constant $C$ independent of $i$ such that
		\begin{align}\label{decayesti}
		u_i(y)\leq C-u_i(0)-4\ln|y|.
		\end{align}
	\end{lemm}
	\begin{proof}
		By Harnack inequality and the above estimate, it suffices to show
		\begin{align*}
		\min_{\partial B_r} u_i(y)\leq C-u_i(0)-4\ln r,\ \forall\ r\geq R_i/\lambda_i,
		\end{align*}
		where $\lambda_i$ and $R_i$ is defined above.
		
		For any $r\geq R_i/\lambda_i$, define $\bar{u}_i(\xi)=u_i(2r\xi)+2\ln r$, $\forall\ \xi\in B_1$. Then $\lambda(A^{\bar{u}_i})\in \Gamma_2$ in $B_1$. Fix some $\xi_i$ such that $|\xi_i|=\frac{1}{2\lambda_i r}$, then
		\begin{align}\label{uibar1}
		\bar{u}_i(0)+2\ln |\xi_i|=-2\ln 2.
		\end{align}
		Moreover, by (\ref{estimateui}),
		\begin{align*}
		\bar{u}_i(\xi_i)+2\ln |\xi_i|=\tilde{u}_i(\frac{\xi_i}{|\xi_i|})-2\ln 2=-4\ln2+o(1).
		\end{align*}
		Therefore, for some $i_0$ independent of $r$, 
		\begin{align}\label{uibar2}
		\bar{u}_i(\xi_i)+2\ln |\xi_i|\leq -3\ln 2\text{ for all }i\geq i_0.
		\end{align}
		Since $|\xi_i|\leq\frac{1}{2R_i}\leq\frac{1}{2}$, by (\ref{uibar1})(\ref{uibar2}) and the B\^ocher type theorem Theorem \ref{Bocher1}, for $w_i=\exp(-\bar{u}_i/4)$,
		\begin{align*}
		\max_{\partial B_{1/2}}w\geq \frac{|w_i(0)-w_i(\xi_i)|}{C|\xi_i|}\geq \frac{1}{C|\xi_i|^{1/2}}=\frac{1}{Cw_i(0)}.
		\end{align*}
		Back to $u_i$, then we obtain (\ref{decayesti}).
	\end{proof}
	Now return to the proof of $C^0$ estimate. Let
	\begin{align*}
	\tilde{K}_i=\frac{1}{\|\nabla^2K_i\|_{C^0(\mathbb{S}^2)}}(K_i\circ\Phi-K_i(P)).
	\end{align*}
	By Kazdan-Warner type identity (\ref{KazdanWarner}),
	\begin{align}\label{KW1}
	0=\frac{1}{\|\nabla^2K_i\|_{C^0(\mathbb{S}^2)}}\int_{\mathbb{R}^2}\partial_l\sigma_2(\lambda(A^{u_i}))e^{u_i}dy=\int_{\mathbb{R}^2}\partial_l\tilde{K}_ie^{u_i}dy,\quad l=1,2,
	\end{align}
	and
	\begin{align}\label{KW2}
	0=\frac{1}{\|\nabla^2K_i\|_{C^0(\mathbb{S}^2)}}\sum_{l=1}^{2}\int_{\mathbb{R}^2}y_l\partial_l\sigma_2(\lambda(A^{u_i}))e^{u_i}dy=\int_{\mathbb{R}^2}y\cdot\nabla\tilde{K}_ie^{u_i}dy.
	\end{align}
	Fix some $r_0>0$. Note that (i) implies that $\tilde{K}_i$ is precompact in $C^2$, and it follows that $|\nabla \tilde{K}_i(y)|=\frac{O(1)}{1+|y|^2}$ on $\mathbb{R}^2$. Thus by (\ref{decayesti}) and (\ref{KW1}),
	\begin{align}\label{est2}
	0=\int_{|y|\leq r_0}\partial_l\tilde{K}_ie^{u_i}dy+O(\lambda_i^{-2}),\quad l=1,2.
	\end{align}
	By (\ref{decayesti}) and (\ref{estimateui}), if $q:\mathbb{R}^2\backslash\{0\}\rightarrow\mathbb{R}$ is a homogeneous function of degree $d\in [0,2)$, then
	\begin{align}\label{limid}
	\lim_{i\rightarrow\infty}\lambda_i^d\int_{|y|\leq r_0}q(y)e^{u_i(y)}dy=\int_{\mathbb{R}^2}\frac{q(z)}{(1+|z|^2)^2}dz.
	\end{align}
	Using Taylor's theorem, we write
	\begin{align*}
	\partial_l\tilde{K}_i(y)=\partial_l\tilde{K}_i(0)+\sum_{p=1}^{2}\partial_{p}\partial_{l}\tilde{K}_i(0)y_p+o_{r_0}(1)|y|\quad \text{for }y\leq|r_0|,
	\end{align*}
	where $o_{r_0}(1)\rightarrow 0$ as $r_0\rightarrow 0$. Plugging this into (\ref{est2}), using also (\ref{limid}),
	\begin{align}\label{taylor1}
	0=M^{(i)}\partial_l\tilde{K}_i(0)+M^{(i)}\sum_{p=1}^{2}\partial_{p}\partial_{l}\tilde{K}_i(0)\mu_p^{(i)}+o_{r_0}(1)\lambda_i^{-1},\quad l=1,2,
	\end{align}
	where $M^{(i)}$ and $\mu_p^{(i)}$ are given by
	\begin{align*}
	M^{(i)}=\int_{|y|\leq r_0}e^{u_i}dy\geq C,
	\end{align*}
	\begin{align}\label{mu_p}
	\mu_p^{(i)}=\frac{1}{M^{(i)}}\int_{|y|\leq r_0}y_pe^{u_i}dy=\frac{o(1)}{\lambda_i}.
	\end{align}
	Therefore, we have
	\begin{align}\label{deriK}
	|\nabla K_i(0)|=\frac{o_{r_0}(1)}{\lambda_i}\quad\text{as }i\rightarrow\infty.
	\end{align}
	For (\ref{KW2}), by the same argument, we have
	\begin{align}\label{taylor2}
	0=M^{(i)}\sum_{l=1}^{2}\partial_l\tilde{K}_i(0)\mu_l^{(i)}+M^{(i)}\sum_{l,p=1}^{2}\partial_{p}\partial_{l}\tilde{K}_i(0)\mu_{lp}^{(i)}+o_{r_0}(1)\lambda_i^{-2},\quad l=1,2,
	\end{align}
	where
	\begin{align}\label{mu_pi}
	\mu_{lp}^{(i)}=\frac{1}{M^{(i)}}\int_{|y|\leq r_0}y_ly_pe^{u_i}dy=\frac{o(1)}{\lambda_i}=\frac{O(1)\delta_{lp}+o(1)}{\lambda_i^2}.
	\end{align}
	Combine (\ref{taylor1}) and (\ref{taylor2}) we obtain
	\begin{align*}
	\sum_{l,p=1}^{2}\partial_{p}\partial_{l}\tilde{K}_i(0)(\mu_{lp}^{(i)}-\mu_{l}^{(i)}\mu_{p}^{(i)})=o_{r_0}(1)\lambda_i^{-2}.
	\end{align*}
	By (\ref{mu_p}) and (\ref{mu_pi}), we obtain
	\begin{align*}
	\Delta_{g_{\mathbb{R}^2}} \tilde{K}_i(0)=o_{r_0}(1)\quad \text{as }i\rightarrow\infty.
	\end{align*}
	This together with (\ref{deriK}) implies $\frac{1}{\|\nabla^2K_i\|_{C^0(\mathbb{S}^2)}}(|\nabla K_i|+|\Delta K_i|)(0)$ converges to $0$ as $i\rightarrow\infty$, which is a contradiction to (i).

\end{proof}

Now evaluate equation (\ref{eqn}) at a maximum point $\bar{x}$ of $u$. In the following, we use $\tilde{C}$ to denote some positive constant depending only on an upper bound of $K$ that may vary from line to line.
\begin{align*}
\tilde{C}\geq K(\bar{x})=\sigma_2(g_u^{-1}A_{g_u})(\bar{x})\geq \sigma_2(e^{-u}g^{-1}A_{g})(\bar{x})=e^{-2u(\bar x)}.
\end{align*}
Hence
\begin{align*}
\max_{\mathbb{S}^2} u\geq -\tilde C.
\end{align*}
By the gradient estimate Theorem \ref{grad2}, $|\nabla u|\leq C$ where $C$ depends only on $\max_{\mathbb{S}^2} u$ and $\|K\|_{C^1(\mathbb{S}^2)}$, so 
\begin{align*}
\min_{\mathbb{S}^2} u\geq  -C.
\end{align*}
The $C^0$ estimate is now proved.
\medskip
\subsection{Degree Theory}
\ 

In this section, we prove Theorem \ref{Niren} using degree theories. The proof is by now standard, so we skip some computations here. For details, see \cite{LNW}.

Fix some $0<\alpha'\leq\alpha<1$. We first assume $K\in C^{2,\alpha}(\mathbb{S}^2)$, and the case $K\in C^2(\mathbb{S}^2)$ can be obtained by approximation.

For $\mu\in[0,1]$, denote $K_\mu=\mu K+(1-\mu)/4$. Consider the equation
\begin{align}\label{sigma2eqn2}
\sigma_2(\lambda(A_{g_v}))=K_{\mu},\quad\lambda(A_{g_v})\in\Gamma_2\quad \text{on }\mathbb{S}^2.
\end{align}
By Theorem \ref{grad2}, Theorem \ref{C2} and Proposition \ref{C0esti}, we can choose $C_1$ sufficiently large such that all solutions to (\ref{sigma2eqn2}) belongs to the set
\begin{align}\label{setO}
\mathcal{O}=\{\tilde{v}\in C^{4,\alpha'}(\mathbb{S}^2):\|\tilde{v}\|_{C^{4,\alpha'}(\mathbb{S}^2)}<C_1,\lambda(A_{g_{\tilde{v}}})\in\Gamma_2\}.
\end{align}

Consider the operator $F_\mu:\mathcal{O}\rightarrow C^{2,\alpha'}(\mathbb{S}^2)$ defined by
\begin{align}\label{Fmu}
F_\mu[v]:=\sigma_2(\lambda(A_{g_v}))-K_\mu,\quad\forall v\ \in\mathcal{O}.
\end{align}
By \cite{L89}, the degree $\deg(F_\mu,\mathcal{O},0)$ is well-defined and independent of $\mu\in (0,1]$. By \cite[Theorem B.1]{L95}, it is also independent of $\alpha'\in (0,\alpha]$. Therefore, it suffices to compute the degree for small $\mu$ and some $\alpha'\in (0,\alpha)$.

For $P\in\mathbb{S}^2$ and $1\leq t<\infty$, we can define a M\"obius transformation on $\mathbb{S}^2$ by sending $y$ to $ty$ where $y$ is the stereographic projection coordinates of points and the projection is performed with $P$ as the north pole to the equatorial
plane of $\mathbb{S}^2$.

For any M\"obius transformation $\varphi:\mathbb{S}^2\rightarrow \mathbb{S}^2$ and a function $v$ defined on $\mathbb{S}^2$, denote
\begin{align*}
T_\varphi(v):=v\circ\varphi+4\ln |J_\varphi|
\end{align*}
where $J_\varphi$ denotes  the Jacobian of $\varphi$.

Let $B$ denote the open unit ball in $\mathbb{R}^3$ and let
\begin{align*}
\mathscr{S}_0=\{v\in C^{4,\alpha'}(\mathbb{S}^2):\int_{\mathbb{S}^2}xe^{v(x)}dv_{g}(x)=0\}.
\end{align*}
Notice that $0\in\mathscr{S}_0$ corresponds to the standard bubble on $\mathbb{S}^2$.

For $w\in\mathscr{S}_0$ and $\xi\in B$, define $\pi(w,\xi)$ to be:
\begin{align*}
\pi(w,\xi)=T_{\varphi_{P,t}^{-1}}(w),\quad\text{where }P=\frac{\xi}{|\xi|}\text{ and }t=(1-|\xi|)^{-1}\text{ when }\xi\neq 0,
\end{align*}
and $\pi(w,0)=w$.

The following lemma implies that $\pi$ gives a parametrization of $C^{4,\alpha'}(\mathbb{S}^2)$ with parameters in $\mathscr{S}_0\times B$.
\begin{lemm}
	The map $\pi:\mathscr{S}_0\times B\rightarrow C^{4,\alpha'}(\mathbb{S}^2)$ is a $C^2$ diffeomorphism.
\end{lemm}

Next lemma gives a property of $\pi$: for every given tubular neighborhood $\pi(\mathcal{N}\times B)$ of $\pi(\{0\}\times B)$, all solutions of (\ref{sigma2eqn2}) belongs to $\pi(\mathcal{N}\times B)$ provided that $\mu$ is sufficiently small.
\begin{lemm}
	Let $0<\alpha'<\alpha<1$. Suppose that $K\in C^{2,\alpha}(\mathbb{S}^2)$ satisfies the nondegeneracy condition (\ref{Knondege}). If $v_{\mu_j}=\pi(w_{\mu_j},\xi_{\mu_j})$ solves (\ref{sigma2eqn2}) for some sequence $\mu_j\rightarrow 0^{+}$, then $\xi_{\mu_j}$ stays in a compact subset of $B$ and
	\begin{align*}
	\lim_{j\rightarrow\infty}\|w_{\mu_j}\|_{C^{4,\alpha'}(\mathbb{S}^2)}=0.
	\end{align*}
\end{lemm}

Let $\mathscr{L}$ be the linearized operator
of $F_\mu[\pi(\cdot,\xi)]$ at $\bar{w}\equiv 0$, with the domain $D(\mathscr{L})$  being the tangent plane to $\mathscr{S}_0$ at $w=\bar w$, i.e. 
\begin{align*}
D(\mathscr{L})=\{\eta\in C^{4,\alpha'}(\mathbb{S}^2):\int_{\mathbb{S}^2}x\eta(x) dv_{g}(x)=0\}.
\end{align*}
By implicit function theorem, $\mathscr{S}_0$ is represented near $0$ as a graph over $D(\mathscr{L})$. It is well-known that $\mathscr{L}$ gives an isomorphism from $D(\mathscr{L})$ to
\begin{align*}
R(\mathscr{L}):=\{f\in C^{2,\alpha'}(\mathbb{S}^2):\int_{\mathbb{S}^2}xf(x) dv_{g}(x)=0\}.
\end{align*}
Let $\Pi$ be a projection from $C^{2,\alpha'}(\mathbb{S}^2)$ to $R(\mathscr{L})$ given by
\begin{align*}
\Pi f(x)=f(x)-\frac{3}{4\pi}x\cdot \int_{\mathbb{S}^2}yf(y) dv_{g}(y).
\end{align*}
The following proposition implies that for every given $\xi\in B$, there exists a unique $w_{\xi,\mu}\in\mathcal{N}$ such that the $\mathscr{S}_0$-component of $F_\mu[\pi(w_{\xi,\mu},\xi)]$ is zero.
\begin{prop}
	Let $0<\alpha'<\alpha<1$. Suppose that $K\in C^{2,\alpha}(\mathbb{S}^2)$ and $F_\mu$ is defined by (\ref{Fmu}). Then for every $s_0\in(0,1)$, there exists a constant $\mu_0\in (0,1]$ and a neighborhood $\mathcal{N}$ of $1\in\mathscr{S}_0$ such that, for every $\mu\in (0,\mu_0]$, and $\xi\in \bar{B}_{s_0}\subset B$, there exists a unique $w_{\xi,\mu}\in\mathcal{N}$ depending smoothly on $(\xi,\mu)$ satisfying
	\begin{align*}
	\Pi(F_\mu[\pi(w_{\xi,\mu},\xi)])=0,
	\end{align*}
	Furthermore, there exists some $C > 0$ such that, for $\mu\in(0, \mu_0]$ and $|\xi|$, $|\xi'|\leq s_0$,
	\begin{align*}
	\|w_{\xi,\mu}\|_{C^{4,\alpha'}(\mathbb{S}^2)}&\leq C\mu\| K-\frac{1}{4}\|_{C^{2,\alpha}(\mathbb{S}^2)}\\
	\|w_{\xi,\mu}-w_{\xi',\mu}\|_{C^{4,\alpha'}(\mathbb{S}^2)}&\leq C\mu|\xi-\xi'|\| K-\frac{1}{4}\|_{C^{2,\alpha}(\mathbb{S}^2)}
	\end{align*}
\end{prop}

Define
\begin{align*}
&\Lambda_{\xi,\mu}:=-\frac{3}{4\pi}\int_{\mathbb{S}^2}F_\mu[\pi(w_{\xi,\mu},\xi)](x)xdv_{g}(x),\\
&G:=\int_{\mathbb{S}^2}K\circ\varphi_{P,t} xdv_{g}(x),\quad\text{where }P=\frac{\xi}{|\xi|}\text{ and }t=(1-|\xi|)^{-1}.
\end{align*}

Instead of solving zeros of $F_\mu$ in $\pi(\mathcal{N}\times B)$, it suffices to solve zeros for the finite dimensional map $\xi\rightarrow\Lambda_{\xi,\mu}$. Next lemma gives the formula of $\deg(\Lambda_{\xi,\mu},B_s,0)$.
\begin{lemm}
	Let $\alpha\in(0,1)$. Suppose that $K\in C^{2,\alpha}(\mathbb{S}^2)$ satisfies the nondegeneracy condition (\ref{Knondege}). Then there exists $\mu_0\in (0,1]$ and $s_0\in (0,1]$ such that for all $\mu\in(0,\mu_0]$ and $s\in [s_0,1)$, the Brouwer degrees $\deg(\Lambda_{\xi,\mu},B_s,0)$ and $\deg(G,B_s,0)$ are well-defined and
	\begin{align*}
	\deg(\Lambda_{\xi,\mu},B_s,0)=\deg(G,B_s,0)=-1+\deg(\nabla K, \text{Crit}_{-}(K)).
	\end{align*}
\end{lemm}
\medskip
The last piece is to prove that $\deg(F_1,\mathcal{O},0)$ is the degree of $\Lambda_{\xi,\mu}$.
\begin{prop}\label{degreefor}
	Let $\alpha\in(0,1)$. Suppose that $K\in C^{2,\alpha}(\mathbb{S}^2)$ satisfies the nondegeneracy condition (\ref{Knondege}). Let $\mathcal{O}$ and $F_1$ be as defined in (\ref{setO}) and (\ref{Fmu}) with $\alpha'=\alpha$. Then
	\begin{align*}
	\deg(F_1,\mathcal{O},0)=-1+\deg(\nabla K, \text{Crit}_{-}(K)).
	\end{align*}
\end{prop}
\medskip
Now we are in the position to complete the proof of Theorem \ref{Niren}.
\begin{proof}
	Estimate (\ref{C2norm}) is given by Theorem \ref{grad2}, Theorem \ref{C2} and Proposition \ref{C0esti}. Under the assumption that $\deg(\nabla K, \text{Crit}_{-}(K))\neq 1$, we now prove the existence of solution to (\ref{eqn}).
	
	Let $K_j$ be a sequence of functions in $C^{2,\alpha}(\mathbb{S}^2)$ converging to $K$ in $C^2$. For $j$ sufficiently large, $K_j$ satisfies (\ref{Knondege}), and $\deg(\nabla K_j, \text{Crit}_{-}(K_j))\neq 1$. By Proposition \ref{degreefor}, there exists $v_j\in C^{4,\alpha}(\mathbb{S}^2)$ solving (\ref{sigma2eqn2}). By Theorem \ref{C0esti} we have
	\begin{align*}
	\|v_j\|_{C^0(\mathbb{S}^2)}\leq C,
	\end{align*}
	By Theorem \ref{grad2}, Theorem \ref{C2} and Evans-Krylov's theorem,
	\begin{align*}
	\|v_j\|_{C^{2,\alpha}(\mathbb{S}^2)}\leq C.
	\end{align*}
	The proof is finished by sending $j\rightarrow\infty$ .
\end{proof}

\appendix
\section{Calculus Lemmas}\label{appendix}

\begin{lemm}\label{Calculus-1}
	Let $a>0$ be a positive number, assume that $g\in [-4a,4a]$ satisfies, for $|\tau|<2a$, $|s|\leq 4a$, $0<\lambda<a$ and $\lambda<|s-\tau|$,
	\begin{align*}
	g(\tau+\frac{\lambda^2(s-\tau)}{|s-\tau|^2})-4\ln \frac{|s-\tau|}{\lambda}\leq g(s).
	\end{align*}
	Then
	\begin{align*}
	|g^\prime(s)|\leq \frac{2}{a},\quad |s|\leq a.
	\end{align*}
\end{lemm}

\begin{proof}
	Let $\displaystyle h=e^g$, then we have, for $|\tau|<2a$, $|s|\leq 4a$, $0<\lambda<a$ and $\lambda<|s-\tau|$,
	\begin{align*}
	(\frac{\lambda}{|s-\tau|})^4h(\tau+\frac{\lambda^2(s-\tau)}{|s-\tau|^2})\leq h(s).
	\end{align*}
	
	Apply \cite[Lemma A.1]{LL2} with $\alpha=4$, we have $\displaystyle|h'(s)|\leq \frac{2}{a}h(s)$, for $|s|\leq a$. The result follows immediately.
\end{proof}

\begin{lemm}\label{Calculus-2}
	Let $a>0$ be a constant and $B_{8a}\subset \mathbb{R}^2$ be the ball of radius $8a$ centered at the origin. Assume that $u\in C^1(B_{8a})$ satisfying
	\begin{align}\label{assumption4}
	u_{x,\lambda}(y)\leq u(y),\quad x\in B_{4a},\quad y\in B_{8a},\quad 0<\lambda<2a,\quad \lambda<|y-x|,
	\end{align}
	where $u_{x,\lambda}$ is as defined in (\ref{uxlambda}).

	Then 
	\begin{align*}
	|\nabla u(x)|\leq \frac{2}{a},\quad |x|<a.
	\end{align*}
\end{lemm}

\begin{proof}
	For $x\in B_{a}$ and $e\in \mathbb{R}^2$, with $|e|=1$, let $h(s)=u(x+se)$. Then, by (\ref{assumption4}), $h$ satisfies the hypothesis of Lemma \ref{Calculus-1}. Thus we have $\displaystyle|h^\prime(0)|\leq \frac{2}{a}$,  i.e. $\displaystyle|\nabla u(x)\cdot e|\leq \frac{2}{a}$. The lemma is now proved.
	
\end{proof}

\begin{lemm}\label{Calculus-3}
	Let $u\in C^1(\mathbb{R}^2)$ satisfy 
	\begin{align*}
	u_{x,\lambda}(y)\leq u(y),  \quad \forall  \lambda>0,  x\in\mathbb{R}^2,  |y-x|\geq\lambda.
	\end{align*}
	Then $u$ must be constant.
\end{lemm}

\begin{proof}
	Let $a\rightarrow \infty$ in Lemma \ref{Calculus-2}, we have $|\nabla u|=0$. Thus $u$ must be constant.
\end{proof}

\end{document}